%% file: main.tex
\documentclass[11pt, a4paper, reqno]{amsart}   
\usepackage[T1]{fontenc}
\usepackage[utf8]{inputenc}
\usepackage{graphicx} 
\usepackage{amsmath,amssymb,amsthm,amsfonts}
\usepackage[margin=1in]{geometry}
\usepackage[foot]{amsaddr}
\usepackage{enumerate}
\usepackage[style=trad-abbrv]{biblatex}

\usepackage[parfill]{parskip}

\usepackage{xcolor}
\usepackage{mathrsfs}
\usepackage{commath}    
\usepackage{mathtools}  
\usepackage{thm-restate}
\usepackage{hyperref}
\hypersetup{
  colorlinks,
  linkcolor={red!50!black},
  citecolor={blue!50!black},
  urlcolor={blue!80!black}
}
\usepackage{tikz}
\usetikzlibrary{shapes}

\usepackage[nameinlink,capitalise]{cleveref}

\newcommand{\N}{\mathbb{N}}
\newcommand{\R}{\mathbb{R}}
\newcommand\ceil[1]{\left\lceil#1\right\rceil}

\newcommand{\CC}{\mathcal{C}}

\newcommand{\LL}{\mathcal{L}}
\renewcommand{\SS}{\mathcal{S}}

\newtheorem{theorem}{Theorem}
\newtheorem{lemma}[theorem]{Lemma}
\newtheorem{conjecture}[theorem]{Conjecture}

\newtheorem{corollary}[theorem]{Corollary}
\newtheorem{claim}{Claim}
\theoremstyle{definition}

\newenvironment{claimproof}[1][\proofname]{%
\begin{proof}[#1]%
}{\end{proof}}

\newcommand\st{\text{ s.t. }}

\let\ge\geqslant
\let\leq\leqslant
\let\geq\geqslant

\let\epsilon\varepsilon

\let\setminus-

\newcommand\dist{\mathrm{dist}}

\title{A Caro-Wei bound for induced linear forests in graphs} 
\author{Gwenaël Joret \and Robin Petit}

\address{
  Computer Science Department\\
  Université libre de Bruxelles\\
  Brussels\\
  Belgium
}

\email{gwenael.joret@ulb.be}
\email{robin.petit@ulb.be}

\thanks{ G.\ Joret is supported by a PDR grant from the Belgian National Fund for Scientific Research (FNRS)}

\date{\today}

\begin{document}

\begin{abstract}
A well-known result due to Caro (1979) and Wei (1981) states that every graph $G$ has an independent set of size at least $\sum_{v\in V(G)} \frac{1}{d(v) + 1}$, where $d(v)$ denotes the degree of vertex $v$. 
Alon, Kahn, and Seymour (1987) showed the following generalization: For every $k\geq 0$, every graph $G$ has a $k$-degenerate induced subgraph with at least $\sum_{v \in V(G)}\min\{1, \frac {k+1}{d(v)+1}\}$ vertices. 
In particular, for $k=1$, every graph $G$ with no isolated vertices has an induced forest with at least $\sum_{v\in V(G)} \frac{2}{d(v) + 1}$ vertices. 
Akbari, Amanihamedani, Mousavi, Nikpey, and Sheybani (2019) conjectured that, if $G$ has minimum degree at least $2$, then one can even find an induced {\em linear} forest of that order in $G$, that is, a forest where each component is a path. 

In this paper, we prove this conjecture and show a number of related results. In particular, if there is no restriction on the minimum degree of $G$, we show that there are infinitely many ``best possible'' functions $f$ such that $\sum_{v\in V(G)} f(d(v))$ is a lower bound on the maximum order of a linear forest in $G$, and we give a full characterization of all such functions $f$. 
\end{abstract}

\maketitle

\section{Introduction}
All graphs in this paper are finite, simple, and undirected. 
The degree of a vertex $v$ in a graph $G$ is denoted $d_G(v)$, and we drop the subscript $G$ when the graph is clear from the context.  

Caro~\cite{caro1979new} and Wei~\cite{wei1981lower} proved the following well-known lower bound on the independence number of a graph $G$ as a function of the degrees of its vertices. 

\begin{theorem}[Caro~\cite{caro1979new} and Wei~\cite{wei1981lower}]
    \label{thm:CaroWei}
    Every graph $G$ has an independent set of size at least $\sum_{v\in V(G)} \frac{1}{d(v) + 1}$. 
\end{theorem}

Note that this lower bound is best possible, as shown by complete graphs.    
Alon, Kahn and Seymour~\cite{alon1987large} generalized this result as follows, where a graph is {\em $k$-degenerate} if every subgraph has a vertex of degree at most $k$. 

\begin{theorem}[Alon, Kahn, and Seymour~\cite{alon1987large}]\label{thm:Alon Kahn Seymour k-degenerate}
For every integer $k\geq 0$, every graph $G$ has an induced subgraph with at least $\sum_{v \in V(G)}\min\left\{1, \frac {k+1}{d(v)+1}\right\}$ vertices that is $k$-degenerate.
\end{theorem}

The bound in \Cref{thm:Alon Kahn Seymour k-degenerate} is best possible, as shown by complete graphs. 
\Cref{thm:Alon Kahn Seymour k-degenerate} with $k=0$ corresponds to the Caro-Wei bound for independent sets. 
In this paper, we focus on the $k=1$ case. 
For $k=1$, \Cref{thm:Alon Kahn Seymour k-degenerate} shows that every graph $G$ with no isolated vertex has an induced forest with at least $\sum_{v \in V(G)}\frac 2{d(v)+1}$ vertices.\footnote{We remark that Punnim~\cite{punnim2003forests} also proved the result for $k=1$, apparently unaware of~\cite{alon1987large}, with a probabilistic proof that differs from the constructive one given in~\cite{alon1987large}.} 
It is natural to ask whether one can guarantee some additional properties for the trees in this induced forest, such as a bound on the maximum degree, or a specific structure. 
This is the topic that we explore in this paper. 

Our starting point is a conjecture of Akbari, Amanihamedani, Mousavi, Nikpey, and Sheybani~\cite{akbari2019maximum}: They conjectured that one can impose an upper bound of $2$ on the maximum degree of the trees---or equivalently, that each tree is a path---{\em provided} that $G$ has minimum degree at least $2$. 


\begin{conjecture}[Akbari {\it et al.}~\cite{akbari2019maximum}]
    \label{conj:Akbari_et_al}
    Every graph $G$ with minimum degree at least $2$ has an induced linear forest with at least $\sum_{v\in V(G)} \frac{2}{d(v) + 1}$ vertices. 
\end{conjecture}

Here, a forest is {\em linear} if every component of the forest is a path. 
(We remark that a single vertex is considered to be a path in this paper, thus a linear forest is allowed to have isolated vertices.)
Note that it is necessary to rule out vertices of degree $1$ in the above conjecture, as shown by the claw $K_{1,3}$.  
As supporting evidence for their conjecture, the authors of~\cite{akbari2019maximum} proved it in the case where the graph $G$ is regular.

Our first contribution is a proof of \cref{conj:Akbari_et_al}, which follows from \cref{thm:proof_of_conj} below. 
In order to state this theorem and subsequent results in this paper, it will be convenient to introduce the following shorthand notation: 
Given a function $f : \mathbb N \to [0, 1]$ and a graph $G$, with a slight abuse of notation we let $f(G)$ denote the following quantity:
\[
f(G) \coloneqq \sum_{v \in V(G)}f(d(v)).
\]

\begin{restatable}{theorem}{thmproofofconj}\label{thm:proof_of_conj}
Let $f: \N \to \R$ be defined as follows: 
\[
f(d) = \begin{cases}
    1 &\text{ if } d = 0\\
    \frac 56 &\text{ if } d=1\\
    \frac 2{d+1} &\text{ if } d\geq 2. 
\end{cases}
\]
Then, every graph $G$ has an induced linear forest with at least $f(G)$ vertices.  
\end{restatable}

Our proof of \cref{thm:proof_of_conj} is a short inductive argument. 
We remark that \cref{thm:proof_of_conj} is best possible in the sense that the value of $f(d)$ cannot be increased for any $d$. 
(This is clear for $d=0$, for $d=1$ this is witnessed by $K_{1,3}$, and for $d\geq 2$ by complete graphs.) 

If we compare \cref{thm:Alon Kahn Seymour k-degenerate} for $k=1$ with \cref{thm:proof_of_conj}, we see that the lower bounds are the same except for the contribution of degree-$1$ vertices, which are respectively $1$ and $5/6$. 
As it turns out, one can keep a contribution of $1$ for degree-$1$ vertices and get an induced forest that is  almost linear, namely, every component is a {\em caterpillar}. 
Here, a caterpillar is any tree that can be obtained from a path $P$ with at least one vertex by adding a (possibly empty) set of new vertices and making each new vertex adjacent to a vertex of $P$. 
(Let us point out that a graph consisting of a single vertex is a caterpillar, that is, a caterpillar does not need to contain an edge.) 

\begin{corollary}\label{cor:caterpillar}
Every graph $G$ with no isolated vertices has an induced forest with at least $\sum_{v\in V(G)} \frac{2}{d(v) + 1}$ vertices, every component of which is a caterpillar.  
\end{corollary}

\cref{cor:caterpillar} follows easily from \cref{thm:proof_of_conj}, the proof consists of first removing the degree-$1$ vertices and then applying \cref{thm:proof_of_conj}. 

It is interesting to rephrase the previous results in terms of treewidth and pathwidth: 
graphs of treewidth $0$ correspond to edgeless graphs, and same for pathwidth $0$. 
Graphs of treewidth $1$ are exactly forests, while graphs of pathwidth $1$ are exactly forests where each component is a caterpillar. 
Thus, the Caro-Wei bound gives the extremal bound as a function of the degrees for finding induced subgraphs with treewidth/pathwidth $0$. 
\cref{thm:Alon Kahn Seymour k-degenerate} with $k=1$ does the same for induced subgraphs with treewidth $1$, and \cref{cor:caterpillar} shows that the extremal function is the same for pathwidth $1$ as for treewidth $1$. 

Now, let us come back to \cref{thm:proof_of_conj}. 
As already mentioned, the function $f$ in  \cref{thm:proof_of_conj} cannot be improved, thus the result is best possible in this sense. 
However, and perhaps surprisingly, there are other ``best possible'' functions $f$ such that $f(G)$ is a lower bound on the maximum order of an induced linear forest in $G$.  
A simple one is the function $f$ with $f(0)=f(1)=1$, $f(2)=2/3$, and $f(d)=0$ for all $d\geq 3$. 
(We skip the easy proof that it is a lower bound; it is best possible as shown by $K_3$ and $K_{1,t}$ for $t\geq 3$.) 
This naturally raises the following question: 
\begin{align*}
&\textrm{\it Given a graph $G$, what is the best possible lower bound on the maximum order} \\[-0.5ex]
&\textrm{\it of an induced linear forest in $G$ as a function of its degree sequence only?} 
\end{align*}
In order to answer this question, it will be helpful to introduce the following notation and definitions. 
First, given a class of graphs $\CC$ and a graph $G$, let $\alpha_{\CC}(G)$ denote the maximum order of an induced subgraph $H$ of $G$ such that $H$ belongs to $\CC$. 
(Above, we focused on the case where $\CC$ is the set of linear forests but we will soon consider other graph classes below, hence the notation.) 
Given a graph class $\CC$ and a function $f : \mathbb N \to [0, 1]$, the function $f$ is called a \emph{lower bound} for the graph invariant $\alpha_{\CC}$ if $\alpha_{\CC}(G) \geq f(G)$ holds for all graphs $G$. 
Such a lower bound $f$ is said to be \emph{extremal} if $f$ cannot be augmented, that is, if there does not exist $g : \mathbb N \to [0, 1]$ with $g \not \equiv 0$ such that $f+g$ is also a lower bound for $\alpha_{\CC}$. 
A lower bound $f$ is {\em dominated} by a lower bound $f'$ if $f(d) \leq f'(d)$ holds for every $d \in \N$. 
We remark that, at first sight, it might not be clear that every lower bound $f$ for $\alpha_{\CC}$ is dominated by some extremal lower bound $f'$---indeed, this requires a proof---however, this will be the case for all graph classes $\CC$ considered in this paper. 

Let $\LL$ denote the set of all linear forests; 
thus, $\alpha_{\LL}(G)$ denotes the maximum order of an induced linear forest in $G$. 
We may now state the following result, which fully answers the question raised above.


\begin{theorem}
\label{thm:characterization of ELBs for linear forests in intro}
Let $\varepsilon \in \R$ with $0 \leq \varepsilon \leq \frac 16$ and define 
\[f_\varepsilon : \mathbb N \to [0, 1] : d \mapsto \begin{cases}1 &\text{ if }d=0\\1-\varepsilon &\text{ if } d=1\\\frac 23 &\text{ if } d=2\\
\min\left\{3\varepsilon, \frac 2{d+1}\right\} &\text{ if } d\geq 3.\end{cases}\]
Then $\alpha_{\LL}(G) \geq f_\varepsilon(G)$ holds for every graph $G$. 
Moreover, this lower bound $f_\varepsilon$ is extremal, and every lower bound for $\alpha_{\LL}$ is dominated by some lower bound of this form. 
\end{theorem}

The optimal choice for $\varepsilon$ given the degree sequence of a graph $G$ is given in~\Cref{thm:epsilon star for k-caterpillars} in \Cref{sec:caterillars_bounded_degree}.

\Cref{thm:characterization of ELBs for linear forests in intro} is the main result of this paper. 
Note that \Cref{thm:proof_of_conj} follows from \Cref{thm:characterization of ELBs for linear forests in intro} by taking $\varepsilon=1/6$. 
However, we were unable to find a short proof for \Cref{thm:characterization of ELBs for linear forests in intro}, the inductive argument that we use in the proof of \Cref{thm:proof_of_conj} only works for $\varepsilon=1/6$. 
The heart of the proof of \Cref{thm:characterization of ELBs for linear forests in intro} is an auxiliary lemma, \Cref{lemma:A-B-C}, which we call the {\em ABC Lemma}. 
Informally, given a tripartition of the vertex set of $G$ into sets $A, B, C$, the lemma provides a large induced linear forest in $G$ satisfying the constraint that vertices in $B$ are leaves and vertices in $C$ are isolated. 
How large is the forest depends on the tripartition, with vertices in $B$ contributing less than those in $A$, and vertices in $C$ less than those in $B$. 

Using the ABC lemma, we were able to generalize \Cref{thm:characterization of ELBs for linear forests in intro} to the setting of induced subgraphs that are forests of caterpillars and have maximum degree at most $k$ for some $k\geq 2$. 
For $k=2$, this corresponds to induced linear forests. 
For $k=+\infty$, this corresponds to induced forests of caterpillars, as in \Cref{cor:caterpillar}. 
Thus, varying $k$ gives a way of interpolating between these two extremes. 
Again, for fixed $k$, we describe all the corresponding extremal functions, see~\Cref{thm:characterization of ELBs for max-degree-k caterpillars}. 
As it turns out, the proof for general $k$ is not more difficult than for $k=2$, which is why we include this result.  

We conclude this introduction by mentioning one last contribution. 
In order to motivate it, let us recall that treewidth is a lower bound on pathwidth, which in turn is a lower bound on treedepth minus $1$. 
These three graph invariants are closely related to each other and play a central role in structural graph theory. 
Graphs of treewidth $1$ are forests, 
graphs of pathwidth $1$ are forests of caterpillars, 
and graphs of treedepth $2$ are forests of {\em stars}, 
where a star is defined here as any tree that can be obtained from a single vertex $v$ by adding a (possibly empty) set of new vertices and making them all adjacent to $v$. 
(Note that this is a special case of the definition of caterpillars above, where the initial path $P$ consists of a single vertex; note also that a graph consisting of a single vertex is considered to be a star.) 
Since the extremal lower bounds are unique and the same in the first two cases, one may wonder if this remains true for the last case as well, that is, for finding induced forests of stars in a given graph.  
However, this is not the case. 
It turns out that there are infinitely many extremal lower bounds, 
as for induced linear forests. 
These are described in the following theorem, where 
$\SS$ denotes the set of all forests of stars.

\begin{restatable}{theorem}{thmforestsofstars}\label{thm:characterization of ELBs for forests of stars}
Let $\varepsilon \in \R$ with $0 \leq \varepsilon \leq \frac 16$ and define 
\[f_\varepsilon : \mathbb N \to [0, 1] : d \mapsto \begin{cases}1 &\text{ if }d=0\\1-\varepsilon &\text{ if } d=1\\
\min\left\{\frac 35,\frac 12+\varepsilon\right\} &\text{ if }d=2\\
\min\left\{\frac 2{d+1}, \frac 1d+\varepsilon\right\} &\text{ if } d \geq 3.\end{cases}\]
Then $\alpha_{\SS}(G) \geq f_\varepsilon(G)$ holds for every graph $G$. 
Moreover, this lower bound $f_\varepsilon$ is extremal, and every lower bound for $\alpha_{\SS}$ is dominated by some lower bound of this form. 
\end{restatable}

The proof of \Cref{thm:characterization of ELBs for forests of stars}  follows a similar high level strategy to that of \Cref{thm:characterization of ELBs for linear forests in intro}. 
In particular, it introduces and uses a so-called {\em AB Lemma} (\Cref{lemma:A-B for forests of stars}) specifically designed for forests of stars. 


The paper is organized as follows. 
In \Cref{sec:linear_forests}, we give the short proof of  \Cref{thm:proof_of_conj}. Then in \Cref{sec:caterillars_bounded_degree} we turn to the proof of \Cref{thm:characterization of ELBs for linear forests in intro}, the main result of this paper. 
The latter proof relies on the ABC Lemma, which is proved in its own section, \Cref{sec:proof_of_ABC_Lemma}. 
Finally, \Cref{sec:star_forests} is devoted to the proof of \Cref{thm:characterization of ELBs for forests of stars}.

\section{Linear forests}
\label{sec:linear_forests} 

Let us start with the proof of \Cref{thm:proof_of_conj}, which we restate here for convenience. 
In the proof, and in the rest of the paper, we use $\Delta(G)$ to denote the maximum degree of a graph $G$. 
Also, given a vertex $v$ in a graph $G$, 
we denote by $N_G(v)$ the set of neighbors of $v$, and we drop the subscript $G$ when the graph is clear from the context.

\thmproofofconj*

\begin{proof}
We proceed by induction on $|V(G)|$. 
Let $\Delta \coloneqq \Delta(G)$. 
First, suppose that $\Delta \leq 2$. 
Then every component of $G$ is a path or a cycle. 
Each path component $P$ contributes at most $|V(P)|$ to $f(G)$ and can be taken entirely. 
Each cycle component $C$ contributes $2|V(C)|/3$ to $f(G)$, and we can take all vertices but one from $C$. 
Note that $|V(C)| - 1 \geq 2|V(C)|/3$ since $|V(C)| \geq 3$. 
Hence, this results in an induced linear forest of $G$ with at least $f(G)$ vertices, as desired. 

Next, assume that $\Delta \geq 3$. 
Observe that:
\begin{equation}
\label{eq:non_increasing_diffs}
f(k-1)-f(k) \geq f(d-1)-f(d)  \qquad \qquad 
\quad \textrm{ for every $k,d$ with } 1 \leq k \leq d
\end{equation}
and 
\begin{equation}
\label{eq:d-1 vs d}
d \cdot (f(d-1)-f(d)) = f(d) \,\phantom{-}\qquad \qquad 
\quad \textrm{ for every $d$ with } d \geq 3.
\end{equation}
(Indeed, the value of $f(1)$ was chosen so that $f(2)-f(3) = f(1)-f(2) = f(0)-f(1) = 1/6$ and \eqref{eq:non_increasing_diffs} holds.) 

Let $v$ be a vertex of maximum degree in $G$, let $G' \coloneqq G-v$ and, applying the inductive hypothesis on $G'$, let $F$ be an induced linear forest in $G'$ with at least $f(G')$ vertices. 
Using \eqref{eq:non_increasing_diffs} and \eqref{eq:d-1 vs d}, we obtain
\[
f(G') - f(G) = \sum_{w \in N_G(v)}\left(f(d_G(w)-1) - f(d_G(w))\right) - f(\Delta) \geq \Delta\left(f(\Delta-1)-f(\Delta)\right) - f(\Delta) = 0
\]
implying that $f(G') \geq f(G)$. 
Therefore, $F$ has the desired size for $G$. 
\end{proof}

\cref{cor:caterpillar} follows from~\cref{thm:proof_of_conj}, as we now explain.

\begin{proof}[Proof of~\cref{cor:caterpillar}] 
Let $L$ be the set of vertices of $G$ with degree $1$ and let $G'\coloneqq G-L$. 
By~\Cref{thm:proof_of_conj}, there exists an induced linear forest $F'$ of $G'$ with at least 
$f(G')$ 
vertices, where $f$ is the function from~\Cref{thm:proof_of_conj}. 
Observe that $d_G(v) \geq d_{G'}(v)$ for every $v \in V(G')$. In particular this implies $f(d_G(v)) \leq f(d_{G'}(v))$ for every $v \in V(G')$ since $f$ is nonincreasing. 
Let $F \coloneqq G[V(F') \cup L]$. Then $F$ is a forest of caterpillars that is induced in $G$. Furthermore, 
\begin{align*}    
\abs {V(F)} &= \abs {L} + \abs {V(F')} \\
&\geq \abs {L} + \sum_{v \in V(G')}f(d_{G'}(v)) \\
&\geq \sum_{v \in L}1 + \sum_{v \in V(G) \setminus L}f(d_G(v)) \\
&=\sum_{v \in V(G)}\frac 2{d_{G}(v)+1}, 
\end{align*}
where the last equality holds because every vertex in $V(G')$ has degree at least $2$ in $G$. 
Therefore, $F$ has the desired size. 
\end{proof}

As mentioned in the introduction, the lower bound $f$ provided in \autoref{thm:proof_of_conj} for induced linear forests is extremal but is not the only extremal lower bound. 
For instance, the following function $f'$ is also a lower bound that is extremal, yet $f$ and $f'$ are not comparable.
\[f' : \mathbb N \to [0, 1] : d \mapsto \begin{cases}1 &\text{ if } d=0\\1 &\text{ if }d=1\\\frac 23 &\text{ if } d=2\\0 &\text{ if } d\geq 3.\end{cases}\]
As it turns out, every lower bound is dominated by some extremal lower bound, and 
there are infinitely many extremal lower bounds. 
A full characterization is given in~\Cref{thm:characterization of ELBs for linear forests in intro}; each extremal lower bound is uniquely determined by the weight it sets to degree-$1$ vertices: This should be at least $5/6$ (as in \autoref{thm:proof_of_conj}) and at most $1$ (as in the function $f'$ above). 
Every intermediate value is feasible, provided the weights of all vertices with degree at least $3$ are adapted in the right way.   

Our proof of~\Cref{thm:characterization of ELBs for linear forests in intro} works with no extra effort for a more general problem, namely that of finding induced forest of caterpillars with maximum degree at most $k$ where $k\geq 2$ is a fixed constant, which we believe is of interest.  
Thus, for $k=2$ this corresponds to induced linear forests. 
We introduce this more general setting in the next section and prove a full characterization of the lower bounds for every fixed $k\geq 2$, see~\Cref{thm:characterization of ELBs for max-degree-k caterpillars} in the next section. 
\Cref{thm:characterization of ELBs for linear forests in intro} follows then by taking $k=2$. 


\section{Caterpillars of bounded degree}
\label{sec:caterillars_bounded_degree}

Given an integer $k$ with $k\geq 2$, let $\CC_k$ denote the set of all forests of caterpillars of maximum degree at most $k$. In particular $\LL = \CC_2$. 
Here is the characterization of extremal lower bounds for the parameter $\alpha_{\CC_k}$.

\begin{theorem}\label{thm:characterization of ELBs for max-degree-k caterpillars}
For every integer $k \geq 2$ and every $\varepsilon\in \R$ with $0 \leq \varepsilon \leq \frac 2{(k+1)(k+2)}$, the following function is an extremal lower bound for $\alpha_{\CC_k}$:
\[
f_{k,\varepsilon} : \mathbb N \to [0, 1] :
d \mapsto \begin{cases}
    1 &\text{ if } d=0\\
    1-\varepsilon &\text{ if } d=1\\
    \frac 2{d+1} &\text{ if } 2 \leq d \leq k\\
    \min\left\{(k+1)\varepsilon, \frac 2{d+1}\right\} &\text{ if } d \geq k+1.
\end{cases}
\]
Furthermore, these functions completely characterize the possible lower bounds: Every lower bound for $\alpha_{\CC_k}$ is dominated by $f_{k,\varepsilon}$ for some $\varepsilon$. 
\end{theorem}

Let us make a few remarks before proving \Cref{thm:characterization of ELBs for max-degree-k caterpillars}.   
As mentioned earlier, \Cref{thm:characterization of ELBs for linear forests in intro} follows from the above theorem by taking $k=2$. 
Note also that for $\varepsilon = \frac 2{(k+1)(k+2)}$, one obtains the following extremal lower bound for  $\alpha_{\CC_k}$: 
\[
f_k : \mathbb N \to [0, 1] : d \mapsto \begin{cases}1 &\text{ if } d=0\\\frac {k(k+3)}{(k+1)(k+2)} &\text{ if } d=1\\\frac 2{d+1} &\text{ if } d\ge 2,\end{cases}
\]
which generalizes~\Cref{thm:proof_of_conj} (for $k=2$). 
By varying $k$ from $2$ to $+\infty$, one can think of the above lower bound $f_k$ for $\alpha_{\CC_k}$ as interpolating between the bounds in~\Cref{thm:proof_of_conj} and in~\Cref{cor:caterpillar}.  

Given the degree distribution of a graph $G$, it is not difficult to find a value of $\varepsilon$ in \Cref{thm:characterization of ELBs for max-degree-k caterpillars} giving the best possible lower bound on $\alpha_{\CC_k}$, that is, such that $f_{k,\varepsilon}(G)$ is maximized. 
This is the content of the following theorem. 

\begin{theorem}\label{thm:epsilon star for k-caterpillars}
Let $G$ be a graph, and for $d \geq 0$ let $n_d$ denote the number of vertices in $G$ with degree $d$. 
Let $\varepsilon_k^* \coloneqq \frac 2{(k+1)(D^*+1)}$, where $D^*$ is the smallest integer $D \geq k+1$ such that $(k+1)\sum_{d=k+1}^Dn_d \geq n_1$ if there is such an integer, otherwise let $\varepsilon_k^* \coloneqq 0$. 

Then, the function $t_{G,k} : [0, \frac 2{(k+1)(k+2)}] \to \mathbb R : \varepsilon \mapsto \sum_{v \in V(G)}f_{k,\varepsilon}(d(v))$ is maximized at $\varepsilon_k^*$.
In other words, $f_{k,\varepsilon_k^*}$ provides the best lower bound that can be achieved on $G$ in \Cref{thm:characterization of ELBs for max-degree-k caterpillars}.
\end{theorem}

\begin{proof}
First observe that $t_{G,k}$ is piecewise linear. Indeed, for every integer $D \geq k+1$, on the interval $I_D = [\frac 2{(k+1)(D+2)}, \frac 2{(k+1)(D+1)}]$ we have:
\[t_{G,k}(\varepsilon) = \left(n_0 + n_1 + \sum_{d=2}^k\frac {2n_d}{d+1} + \sum_{d=D+1}^{\Delta(G)}\frac {2n_d}{d+1}\right) + \varepsilon\left((k+1)\sum_{d=k+1}^Dn_d - n_1\right).\]

Thus, the linear components of $t_{G,k}$ consist of the interval $[0, \frac 2{(k+1)(\Delta(G)+1)}]$ and the intervals $I_D$ for $k+1 \leq D \leq \Delta(G)-1$, and $t_{G,k}$ is monotone on these intervals. Hence, it is enough to look at the values $t_{G,k}(0)$ and $t_{G,k}(\frac 2{(k+1)(D+1)})$ for $D \geq k+1$.

If $(k+1)\sum_{d=k+1}^{\Delta(G)}n_d < n_1$, then for every $D \geq k+1$ and $\varepsilon = \frac 2{(k+1)(D+1)}$:
\[t_{G,k}(0) - t_{G,k}(\varepsilon) = n_1\varepsilon - \sum_{d=k+1}^{\Delta(G)} n_df_{k,\varepsilon}(d) \geq n_1\varepsilon - \sum_{d=k+1}^{\Delta(G)}n_d(k+1)\varepsilon  > 0,\]
and so $t_{G,k}$ is maximized at $\varepsilon = \varepsilon_k^* = 0$. 

If $(k+1)\sum_{d=k+1}^{\Delta(G)}n_d \geq n_1$, for $D \geq k+1$ define
\[T_{G,k}(D) \coloneqq t_{G,k}\left(\frac 2{(k+1)(D+1)}\right) - t_{G,k}\left(\frac 2{(k+1)(D+2)}\right).\]
Now, since
\begin{align*}
    T_{G,k}(D) &= \frac {2n_1}{k+1}\left(\frac 1{D+2}-\frac 1{D+1}\right) + \sum_{d=k+1}^D2n_d\left(\frac 1{D+1} - \frac 1{D+2}\right) \\
    &= \frac 2{(k+1)(D+1)(D+2)}\left((k+1)\sum_{d=k+1}^Dn_d - n_1\right),
\end{align*}
we see that $T_{G,k}(D) \geq 0$ for every $D \geq D^*$, and conversely, $T_{G,k}(D) < 0$ for every $D < D^*$. Therefore, $t_{G,k}(\varepsilon)$ is maximized at $\varepsilon = \varepsilon_k^* = \frac 2{(k+1)(D^*+1)}$.
\end{proof}

\Cref{thm:characterization of ELBs for max-degree-k caterpillars} provides an exact characterization of all the lower bounds for $\alpha_{\CC_k}$ that only depend on the degree sequence of the graph $G$ under consideration. 
That is, knowing only the degree sequence of $G$, it is not possible to give a better lower bound for $\alpha_{\CC_k}$ than those described in the theorem. 
However, if we know some extra local information about the degree-$1$ vertices of $G$, namely the degrees of their neighbors, it is possible to state a more precise lower bound on $\alpha_{\CC_k}$ that
dominates all the bounds provided in \Cref{thm:characterization of ELBs for max-degree-k caterpillars}.
To state this lower bound, we need to introduce the following function $h_{k,G} : V(G) \to [0,1]$ where $G$ is a graph and $k$ is an integer with $k \geq 2$: 
\[h_{k,G}(v) \coloneqq \begin{cases}
    1 &\text{ if } d(v) = 0\\
    1 &\text{ if } d(v) = 1 \text{ and } d(w) \leq k \text{ where $N(v) = \{w\}$}\\
    1-\frac 2{(k+1)(d(w)+1)} &\text{ if } d(v) = 1 \text{ and } d(w) \geq k+1 \text{ where $N(v) = \{w\}$}\\
    \frac 2{d(v)+1} &\text{ if } d(v) \geq 2.
\end{cases}\]
Here is the refined lower bound on $\alpha_{\CC_k}$. 

\begin{theorem}\label{thm:k-caterpillar without epsilon}
For every integer $k \geq 2$, every graph $G$ has an induced forest of caterpillars of maximum degree at most $k$ with at least $\sum_{v \in V(G)}h_{k,G}(v)$ vertices.
\end{theorem}

\Cref{thm:k-caterpillar without epsilon} is the main technical result of the paper. 
\Cref{thm:characterization of ELBs for max-degree-k caterpillars} follows easily from \Cref{thm:k-caterpillar without epsilon}, as we now explain. 
From a technical point of view, \Cref{thm:k-caterpillar without epsilon} can be thought of as a ``local strengthening'' of \Cref{thm:characterization of ELBs for max-degree-k caterpillars} to help the proof by induction go through. 
In the following proof, and in the rest of the paper, 
a {\em leaf} of a graph $G$ is any vertex $v$ with degree exactly $1$. (Thus, two adjacent vertices of degree $1$ are both considered to be leaves.)

\begin{proof}[Proof of \Cref{thm:characterization of ELBs for max-degree-k caterpillars} assuming \Cref{thm:k-caterpillar without epsilon}]
Given an integer $k \geq 2$ and a real number $\varepsilon$ with $0 \leq \varepsilon \leq \frac 2{(k+1)(k+2)}$, we need to show that $f_{k,\varepsilon}$ is an extremal lower bound for $\alpha_{\CC_k}$. 
First, we show that it is a lower bound (using \Cref{thm:k-caterpillar without epsilon}), and then we construct graphs showing that it is extremal. 

To show that $f_{k,\varepsilon}$ is a lower bound for $\alpha_{\CC_k}$, consider any graph $G$. 
We need to show that $\alpha_{\CC_k}(G) \geq \sum_{v \in V(G)}f_{k,\varepsilon}(v)$. 
Clearly, it is enough to prove it in the case where $G$ is connected. 
Furthermore, the inequality clearly holds if $|V(G)| \leq 2$, thus we may assume that $|V(G)| \geq 3$. 
In particular, $\Delta(G) \geq 2$.   

Let $L$ be the set of leaves of $G$.  
(Recall that a {\em leaf} is defined as a vertex of degree $1$.) 
Given a vertex $v$ of $G$ with degree at least $2$, let $\ell(v)$ denote the number of neighbors of $v$ that are leaves, and let 
\[D(v) \coloneqq \left(h_{k,G}(v) + \sum_{w \in N(v) \cap L}h_{k,G}(w)\right) - \left(f_{k,\varepsilon}(d(v)) + \sum_{w \in N(v) \cap L}f_{k,\varepsilon}(d(w))\right).\]
Since every leaf is adjacent to a non-leaf vertex (since $G$ is connected and $|V(G)| \geq 3$), it follows that 
\[
\sum_{v \in V(G)}(h_{k,G}(v) - f_{k,\varepsilon}(d(v))) = \sum_{v \in V(G)-L}D(v).
\]

Consider some vertex $v$ with $d(v) \geq 2$. 
If $2 \leq d(v) \leq k$, then
\[D(v) = \left(\frac 2{d(v)+1}+\ell(v)\right) - \left(\frac 2{d(v)+1} + \ell(v) \cdot (1-\varepsilon)\right) = \ell(v)\varepsilon \geq 0.\]

If on the other hand $d(v) \geq k+1$, then
\[f_{k,\varepsilon}(d(v)) + \sum_{w \in N(v) \cap L}f_{k,\varepsilon}(d(w)) = \ell(v) + \min\left\{(k+1)\varepsilon, \frac 2{d(v)+1}\right\} - \ell(v)\varepsilon.\]
Thus, either $(k+1)\varepsilon \leq \frac 2{d(v)+1}$, in which case:
\[D(v) = \frac {2(k+1-\ell(v))}{(k+1)(d(v)+1)} - (k+1-\ell(v))\varepsilon \geq \frac {2(k+1-\ell(v))}{(k+1)(d(v)+1)} - \frac {2(k+1-\ell(v))}{(k+1)(d(v)+1)} = 0,\]
or $(k+1)\varepsilon>\frac 2{d(v)+1}$, in which case:
\[D(v) = \sum_{w \in N(v) \cap L}\left(h_{k,G}(w)-f_{k,\varepsilon}(d(w))\right) = \ell(v)\left(\varepsilon - \frac 2{(k+1)(d(v)+1)}\right) \geq 0.\]

Therefore, in all possible cases for $v$ we have $D(v) \geq 0$. 

It follows that
\[\sum_{v \in V(G)}h_{k,G}(v) = \sum_{v \in V(G)}f_{k,\varepsilon}(d(v)) + \sum_{v \in V(G)-L}D(v) \geq \sum_{v \in V(G)}f_{k,\varepsilon}(d(v)).\]

Since $\alpha_{\CC_k}(G) \geq \sum_{v \in V(G)}h_{k,G}(v)$ by \Cref{thm:k-caterpillar without epsilon}, this concludes the proof that $f_{k,\varepsilon}$ is a lower bound for $\alpha_{\CC_k}$.

Now it remains to show that (i) every lower bound for $\alpha_{\CC_k}$ is bounded by $f_{k,\varepsilon}$ for some $\varepsilon$ satisfying $0 \leq \varepsilon \leq \frac 2{(k+1)(k+2)}$, and (ii) that these bounds are all extremal. 
Let us first show (i). 
Let $\varphi : \mathbb N \to \mathbb R$ be a lower bound for $\alpha_{\CC_k}$ and define $\varepsilon \coloneqq 1-\varphi(1)$.
Clearly, 
\[
\varphi(d) \leq 1  \qquad  \textrm{ for every $d \geq 0$}.
\]
Also, 
\[
\varphi(d) \leq \frac 2{d+1}  \qquad  \textrm{ for every $ d \geq 2$}, 
\]
since
\[(d+1)\varphi(d) = \varphi(K_{d+1}) \leq \alpha_{\CC_k}(K_{d+1}) = 2.\]
If $\varepsilon \geq \frac 2{(k+1)(k+2)}$, then $(k+1)\varepsilon \geq \frac{2}{k+2} \geq \frac{2}{d+1}$ for all $d \geq k+1$. 
We deduce that $\varphi \leq f_{k,\frac 2{(k+1)(k+2)}}$, and we are done. 

Now, assume that $\varepsilon < \frac 2{(k+1)(k+2)}$. 
We claim that  $\varphi$ is dominated by $f_{k,\varepsilon}$.
Note that 
\begin{itemize}
    \item $\varphi(0) \leq 1 = f_{k,\varepsilon}(0)$;
    \item $\varphi(1) = 1-\varepsilon = f_{k,\varepsilon}(1)$, and
    \item $\varphi(d) \leq \frac 2{d+1} = f_{k,\varepsilon}(d)$ for $2 \leq d \leq k$.
\end{itemize}
It remains to show that $\varphi(d) \leq f_{k,\varepsilon}(d)$ for all $d \geq k + 1$. To do so, 
let $n \geq 1$, and define the graph $H_{n,k}$ as follows: 
\[
V(H_{n,k}) = \{(v, i) : 1 \leq v \leq n, 0 \leq i \leq k+1\}, 
\]
and there is an edge between vertex $(v, i)$ and vertex $(w, j)$ in $H_{n,k}$ if and only if $v=w$ and $i=0$, or $v \neq w$ and $i = j = 0$. Informally, $H_{n,k}$ is the graph obtained by starting with $K_n$ (whose vertices are $\{(v, 0) : 1 \leq v \leq n\}$) 
and adding exactly $k+1$ leaves to each vertex. 

We claim that $\alpha_{\CC_k}(H_{n,k}) = (k+1)n$. It is clear that for every $1 \leq v \leq n$, at most $k+1$ of the vertices $(v, 0), \ldots, (v, k+1)$ can be in any induced forest of maximum degree at most $k$ in $H_{n,k}$, otherwise $(v, 0)$ must be in it and must have degree at least $k+1$, hence the inequality $\alpha_{\CC_k}(H_{n,k}) \leq (k+1)n$. This upper bound holds with equality, as witnessed by the forest induced by the set $\{(v, i) : 1 \leq v \leq n, 1 \leq i \leq k+1\}$.

We deduce that the function $\varphi$ must satisfy
\[n\varphi(n+k) + (k+1)n\varphi(1) 
\leq \alpha_{\CC_k}(H_{n,k}) = (k+1)n.\]
In particular, for $d=n+k \geq k+1$:
\[\varphi(d) \leq \frac 1n\left((k+1)n\left(1-\varphi(1)\right)\right) = (k+1)\varepsilon.\]
Therefore, $\varphi(d) \leq \min\{(k+1)\varepsilon, \frac 2{d+1}\} = f_{k,\varepsilon}(d)$ holds for all $d\geq k+1$, and we conclude that $\varphi \leq f_{k,\varepsilon}$.

It remains to show property (ii), stating that $f_{k,\varepsilon}$ is extremal for every $\varepsilon$. 
Let $\varepsilon_1$ and $\varepsilon_2$ be such that $0 \leq \varepsilon_1 < \varepsilon_2 \leq \frac 2{(k+1)(k+2)}$. 
Then
\[
f_{k,\varepsilon_1}(1) = 1-\varepsilon_1 > 1-\varepsilon_2 = f_{k,\varepsilon_2}(1) 
\]
and
\[
f_{k,\varepsilon_1}(k+1) = (k+1)\varepsilon_1 < (k+1)\varepsilon_2 = f_{k,\varepsilon_2}(k+1).
\]
We deduce that neither of $f_{k,\varepsilon_1}$ and $f_{k,\varepsilon_2}$ dominates the other. 
It follows that $f_{k,\varepsilon}$ is extremal for every $\varepsilon$ with $0 \leq \varepsilon\leq \frac 2{(k+1)(k+2)}$. 
\end{proof}

The key to the proof of \Cref{thm:k-caterpillar without epsilon} is the following technical lemma, which we call the ABC lemma. 

\begin{lemma}[ABC lemma]\label{lemma:A-B-C}
Let $G$ be a graph and let $(A, B, C)$ be a partition of its vertex set (with some parts possibly empty).  
Then $G$ contains an induced linear forest $F$ satisfying
\begin{itemize}
    \item $d_F(v) \leq 2$ for all vertices $v \in V(F) \cap A$, 
    \item $d_F(v) \leq 1$ for all vertices $v \in V(F) \cap B$, 
    \item $d_F(v) \leq 0$ for all vertices $v \in V(F) \cap C$, and
    \item $\abs{V(F)} \geq f(G, A, B, C) \coloneqq \sum_{v \in V(G)}f(v; G, A, B, C)$,
\end{itemize}
where
\begin{align*}  
f(v; G, A, B, C)   &= \begin{cases}f_A(d_G(v))\phantom{1} &\text{ if } v \in A\phantom{d = 2}\\f_B(d_G(v)) &\text{ if } v \in B\\f_C(d_G(v)) &\text{ if } v \in C\end{cases} \qquad\qquad
f_A(d) = \begin{cases}1\phantom{f_A(d(v))} &\text{ if } d  = 0\\\frac 56 &\text{ if } d = 1\\\frac 2{d+1} &\text{ if } d \geq 2\end{cases}\\[1ex]
f_B(d) &= \begin{cases}1\phantom{f_A(d(v))} &\text{ if } d = 0\\\frac 56 &\text{ if } d=1\\\frac 13 &\text{ if } d = 2\phantom{v \in A}\\\frac 4{3(d+1)} &\text{ if } d \geq 3\end{cases}\qquad\qquad
f_C(d) = \begin{cases}1\phantom{f_A(d(v))} &\text{ if } d = 0\\\frac 16 &\text{ if } d=1\\\frac 16 &\text{ if } d=2\\\frac {2}{3(d+1)} &\text{ if } d \geq 3.\end{cases}
\end{align*}
\end{lemma}

Before proving \Cref{thm:k-caterpillar without epsilon} using \Cref{lemma:A-B-C}, let us point out that the values of $f_A(d)$ are best possible, as already discussed before. 
The values of $f_B(d)$ and $f_C(d)$ for $d=0,1,2$ are best possible too. 
This is clear for $f_B(0)$, $f_B(1)$, and $f_C(0)$. 
For $f_B(2), f_C(1), f_C(2)$, this is shown by the examples in \Cref{fig:ABC:tight bounds}.  
(For $d\geq 3$, the values of $f_B(d)$ and $f_C(d)$ in  \Cref{lemma:A-B-C} are most likely not best possible but they are good enough for our purposes.) 

\begin{figure}[!ht]
\centering
\begin{tikzpicture}[-,A/.style={draw,circle,thick},B/.style={draw,rectangle,thick},C/.style={draw,diamond,thick}]
\node[A] (K2-A) at (5.5, 0) {};
\node[C] (K2-C) at (6.5, 0) {};
\path (K2-A) edge (K2-C);

\node[C] (K3-C1) at (11, 0) {};
\node[C] (K3-C2) at (12, 0) {};
\node[A] (K3-A) at (11.5, .866) {};
\path (K3-C1) edge (K3-C2) edge (K3-A);
\path (K3-A) edge (K3-C2);

\node[A] (P3-A1) at (0, 0) {};
\node[B] (P3-B) at (1, 0) {\phantom{x}};
\node[A] (P3-A2) at (2, 0) {};
\path (P3-B) edge (P3-A1) edge (P3-A2);
\node[anchor=west] (annotK2) at (4, -1) {$f_C(1) \leq 1-f_A(1) = \frac 16$};
\node[anchor=west] (annotK3) at (9.25, -1) {$f_C(2) \leq \frac 12(1-f_A(2)) = \frac 16$};
\node[anchor=west] (annotP3) at (-1.5, -1)  {$f_B(2) \leq 2(1-f_A(1)) = \frac 13$};
\end{tikzpicture}
\caption{Examples showing that the values of $f_B(2)$, $f_C(1)$, and $f_C(2)$ in \Cref{lemma:A-B-C} are best possible. Vertices in $A$ are shown as circles, vertices in $B$ are shown as squares, and vertices in $C$ as diamonds.\label{fig:ABC:tight bounds}}
\end{figure}
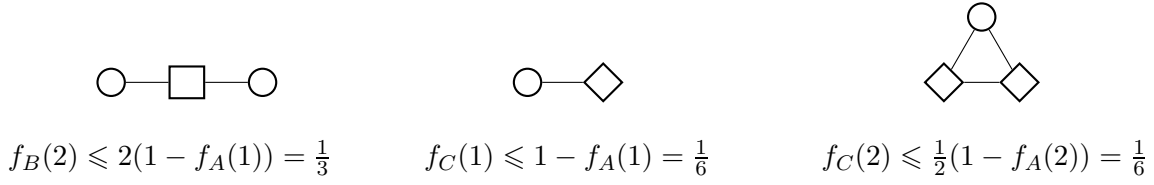

In order to lighten the notation somewhat, we will abbreviate $f(v; G, A, B, C)$ into $f(v; G)$ in the proofs when the partition $(A, B, C)$ of $V(G)$ is clear from the context. 

Now, let us show that \Cref{lemma:A-B-C} implies 
\Cref{thm:k-caterpillar without epsilon}. 
The next section will be devoted to the proof of  \Cref{lemma:A-B-C}. 

\begin{proof}[Proof of \Cref{thm:k-caterpillar without epsilon} assuming \Cref{lemma:A-B-C}]
Let $k \geq 2$ and let $G$ be a graph. 
We need to show that $G$ contains an induced forest of caterpillars of maximum degree at most $k$ with at least $\sum_{v \in V(G)}h_{k,G}(v)$ vertices. 
The proof is by induction on $|V(G)|$, with the base case $|V(G)| \leq 2$ being easily seen to be true.  

For the inductive part, suppose that $|V(G)| \geq 3$ and that the theorem holds for graphs with a smaller number of vertices. 
If $G$ contains a vertex $u$ adjacent to at least $k+1$ leaves, then 
$\sum_{v \in V(G-u)}h_{k,G-u}(v) \geq \sum_{v \in V(G)}h_{k,G}(v)$, and we are done by the inductive hypothesis applied to $G-u$. 
Thus, we may assume that there is no such vertex in $G$. 
Similarly, if $G$ is not connected, we are done by applying the inductive hypothesis on the components of $G$, so we may assume that $G$ is connected. 

Let $L$ be the set of leaves of $G$. Since $G$ is connected and $|V(G)|\geq 3$, no leaf is adjacent to another leaf; that is, $L$ is an independent set. 
We also remark that $L$ could possibly be empty. 

Define the following sets:
\begin{align*}
    A' &\coloneqq \{v \in V(G) \setminus L \st \abs {N(v) \cap L} \leq k-2\} \\
    B' &\coloneqq \{v \in V(G) \setminus L \st \abs {N(v) \cap L} = k-1\} \\
    C' &\coloneqq \{v \in V(G) \setminus L \st \abs {N(v) \cap L} = k\}.
\end{align*}

Applying \Cref{lemma:A-B-C} on $G' \coloneqq G - L$ with the partition $(A', B', C')$ of $V(G')$, we obtain an induced linear forest $F'$ in $G'$ satisfying the degree constraints of the lemma. Now define $F \coloneqq G[V(F') \cup L]$. Clearly, $F$ is an induced forest of caterpillars in $G$ satisfying $\Delta(F) \leq k$. 
We will show that $F$ has the desired number of vertices.

For a vertex $v$ of $G$, let $\ell(v)$ denote the number of neighbors of $v$ in $G$ that are leaves. 
Note that every vertex $v \in V(G')$ satisfies $d_G(v) \geq 2$ and thus $h_{k,G}(v) = \frac 2{d_G(v)+1}$. 
For such a vertex $v$, define 
\[D(v) \coloneqq \left(\ell(v) + f(v;G')\right) - \left(h_{k,G}(v) + \sum_{w \in N(v) \cap L}h_{k,G}(w)\right).\]

We claim that $D(v) \geq 0$ holds for every vertex  $v \in A' \cup B' \cup C'$. 
To show this, let us fix some vertex  $v \in V(G')$. 

{\bf Case~1: $d_G(v) \leq k$.}  
Then $h_{k,G}(v) = \frac 2{d_G(v)+1}$ and $h_{k,G}(w) = 1$ for every vertex $w \in N(v) \cap L$. Thus, $D(v) = f(v;G') - \frac 2{d_G(v)+1}$.

If $v \in A'$, then $D(v) = f_{A}(d_{G'}(v)) - \frac 2{d_G(v)+1} \geq 0$ since $d_{G'}(v) \leq d_G(v)$.

If $v \in B'$, then $\ell(v)=k-1 \geq 1$, so $d_{G'}(v) \leq 1$ and $f_{B}(d_{G'}(v)) \geq \frac 56$, implying $D(v) \geq \frac 56 - \frac 2{d_G(v)+1} > 0$. 

If $v \in C'$, then $d_G(v)=\ell(v)=k$, $N(v) \subseteq L$ and thus $D(v) = 1 - \frac 2{d_G(v)+1} > 0$. 

{\bf Case~2: $d_G(v) \geq k+1$.} 
Then $h_{k,G}(w) = 1-\frac 2{(k+1)(d_G(v)+1)}$ for every vertex $w \in N(v) \cap L$, and thus
\[D(v) = f(v;G') - \frac 2{d_G(v)+1}\frac {k+1-\ell(v)}{k+1}.\]

If $v \in A'$, then $f(v;G')=f_{A}(d_{G'}(v)) \geq \frac 2{d_G(v)+1}$, and thus $D(v) \geq 0$ since $\ell(v) \geq 0$.

If $v \in B'$, then $\ell(v) = k-1$, implying
\[D(v) = f_{B}(d_{G'}(v)) - \frac 2{d_G(v)+1}\frac 2{k+1} \geq f_{B}(d_{G'}(v)) - \frac 4{3(d_G(v)+1)} \geq f_{B}(d_{G'}(v)) - \frac 13.\] 

Observe that $d_{G'}(v) \geq 2$. 
If $d_{G'}(v) = 2$ then $f_{B'}(d_{G'}(v)) = \frac 13$ and thus 
$D(v) \geq 0$. 
If $d_{G'}(v) \geq 3$ then $f_{B'}(d_{G'}(v)) = \frac 4{3(d_{G'}(v)+1)}$ and hence $D(v) \geq \frac 4{3(d_{G'}(v)+1)} - \frac 4{3(d_G(v)+1)} > 0$.

If $v \in C'$, then $\ell(v) = k$, and thus
\[D(v) = f_{C}(d_{G'}(v)) - \frac 2{d_G(v)+1}\frac 1{k+1} \geq f_{C}(d_{G'}(v)) - \frac 2{3(d_G(v)+1)} \geq f_{C}(d_{G'}(v)) - \frac 16.\]

If $d_G(v) \leq k+2$, then $d_{G'}(v) \leq 2$ and hence $f_{C}(d_{G'}(v)) \geq \frac 16$, so $D(v) \geq 0$. 
If $d_G(v) \geq k+3$, then $D(v) \geq f_{C}(d_{G'}(v)) - \frac 2{3(d_G(v)+1)} \geq \frac 2{3(d_{G'}(v)+1)} - \frac 2{3(d_G(v)+1)} > 0.$

Therefore, $D(v) \geq 0$ holds in all possible cases for vertex $v$, and the claim holds. 

Using our claim, we may lower bound the number of vertices in $F$ as follows
\[\abs {V(F)} = \abs L + \abs {V(F')} 
\geq \abs L + f(G') 
= h_{k,G}(G) + \sum_{v\in V(G')} D(v)
\geq h_{k,G}(G),\] 
and this concludes the proof. 
\end{proof}

\input{proof_ABC}

\input{forests_of_stars}

\section*{Acknowledgments}

We thank Louis Esperet for bringing \Cref{conj:Akbari_et_al} to our attention and for his comments on a preliminary version of the manuscript. 
We are grateful to the two anonymous referees for their helpful comments on an earlier version of the manuscript.

\printbibliography

\end{document}

%% file: proof_ABC.tex
\section{Proof of the ABC Lemma}
\label{sec:proof_of_ABC_Lemma}

In this section, we prove \Cref{lemma:A-B-C}. 
Recall that parts of the partition $(A, B, C)$ of $V(G)$ in the lemma are allowed to be empty. 

\begin{proof}[Proof of \Cref{lemma:A-B-C}]
Given an induced linear forest $F$ of a graph $G$, we say that $F$ {\em respects} a partition $(A, B, C)$ of $V(G)$ if $F$ satisfies the degree bounds in the lemma w.r.t.\ the partition $(A, B, C)$.  

The proof of the lemma is by contradiction. 
Suppose thus that the lemma is false, and let  $G$ be a graph and let $(A, B, C)$ be a partition of its vertex set that together form a counterexample to the lemma with a minimum number of vertices. 
Thus, $G$ does not contain an induced linear forest $F$ respecting $(A, B, C)$ that has at least $f(G, A, B, C)$ vertices. 
Clearly, the minimality of $G$ implies that $G$ is connected. 
Furthermore, the lemma is easily seen to hold for $|V(G)|=1,2$, so we also have $|V(G)|\geq 3$. 

Let us introduce the following notation: Given a nonnegative integer $i$, we let $A_i$  denote the set of vertices in $A$ having degree $i$ in $G$. 
The sets $B_i$ and $C_i$ are defined similarly with respect to $B$ and $C$. 

In the proof, we will move around vertices between the three sets $A, B, C$. 
We see these sets as `ranks', with $A$ being the highest rank and $C$ the lowest. 
\emph{Promoting} a vertex $v$ consists in moving $v$ from $B$ to $A$, or from $C$ to $B$. 
Likewise, \emph{demoting} $v$ consists in moving $v$ from $A$ to $B$, or from $B$ to $C$.  

Note that every vertex $v$ of $G$ has degree at least $1$. 
We define the \emph{gain} $\gamma(v)$ of a vertex $v$ to be the increase of $f(v;G)$ when $d(v)$ is decreased by one, that is, 
\[\gamma(v) \coloneqq \begin{cases}f_A(d(v)-1)-f_A(d(v)) &\text{ if } v \in A\\f_B(d(v)-1)-f_B(d(v)) &\text{ if } v \in B\\f_C(d(v)-1)-f_C(d(v)) &\text{ if } v \in C.\end{cases}\]
Observe that $\gamma(v) \geq 0$ always holds, since $f_A(d)$, $f_B(d)$, and $f_C(d)$ are nonincreasing functions. 
Given that the values of the gains for vertices with small degrees will be repeatedly used in the proofs, we summarize these values in \Cref{tab:gains} for reference. 

\begin{table}[ht!]
    \centering
    \begin{tabular}{c|c|c|c}
         & $v\in A$ & $v\in B$ & $v\in C$  \\
         \hline          
        $d = 1$ &  $\frac 16$ & $\frac 16$ & $\frac 56$ \\
        $d = 2$ &  $\frac 16$ & $\frac 12$ & $0$ \\
        $d = 3$ &  $\frac 16$ & $0$ & $0$ \\
        $d = 4$ &  $\frac {1}{10}$ & $\frac {1}{15}$ & $\frac {1}{30}$ \\
        $d \geq 5$ &  $\frac {2}{d(d+1)}$ & $\frac {4}{3d(d+1)}$ & $\frac {2}{3d(d+1)}$ 
    \end{tabular}

    $ $

    $ $
    
    \caption{Value of $\gamma(v)$ for a vertex $v$ of degree $d$.}
    \label{tab:gains}
\end{table}

Similarly, we define the \emph{loss} $\lambda(v)$ of vertex $v$ to be the decrease in $f(v;G)$ when $d(v)$ is increased by one, that is, 
\[\lambda(v) \coloneqq \begin{cases}f_A(d(v))-f_A(d(v)+1) &\text{ if } v \in A\\f_B(d(v))-f_B(d(v)+1) &\text{ if } v \in B\\f_C(d(v))-f_C(d(v)+1) &\text{ if } v \in C.\end{cases}\]
Here also, $\lambda(v) \geq 0$ always holds. 
\Cref{tab:losses} gives a summary of the possible values for $\lambda(v)$. 

\begin{table}[ht!]
    \centering
    \begin{tabular}{c|c|c|c}
         & $v\in A$ & $v\in B$ & $v\in C$  \\
         \hline                  
        $d = 1$ &  $\frac 16$ & $\frac 12$ & $0$ \\
        $d = 2$ &  $\frac 16$ & $0$ & $0$ \\
        $d = 3$ &  $\frac {1}{10}$ & $\frac {1}{15}$ & $\frac {1}{30}$ \\
        $d \geq 4$ &  $\frac {2}{(d+1)(d+2)}$ & $\frac {4}{3(d+1)(d+2)}$ & $\frac {2}{3(d+1)(d+2)}$ 
    \end{tabular}

    $ $

    $ $
    
    \caption{Value of $\lambda(v)$ for a vertex $v$ of degree $d$.}
    \label{tab:losses}
\end{table}

With these notations in hand, we may now turn to the proof, which is split into several claims. 
We remark that in the proofs we often consider other (smaller) graphs derived from $G$; let us emphasize that the gain $\gamma(v)$ and loss $\lambda(v)$ of a vertex $v$ must always be interpreted w.r.t.\ the original graph $G$ and partition $(A, B, C)$ (that is, as they are defined above). 

Here is a quick outline of the proof. 
Ideally, we would have liked to use the same approach as in the proof of \Cref{thm:proof_of_conj}. 
However, this is not possible here because of the existence of vertices $v$ not satisfying $f(v; G) = d(v)\gamma(v)$.  
Instead, we will show that there is a special vertex $v^*$ such that all its neighbors have gain at least that of $v^*$, except perhaps for one neighbor that has zero gain.  
We will then use this vertex $v^*$ to derive a contradiction regarding the minimality of our counterexample.
The proof is organized as follows:
\begin{itemize}
    \item \Cref{claim:ABC:weight is bigger than gain of neighbors} to \Cref{claim:ABC:no C_3--C_3 edge} establish general properties of the graph $G$,
    \item these claims are then used to define $v^*$,
    \item \Cref{claim:ABC:properties of v^*} to \Cref{claim:ABC:v^* has no neighbor in B3} show various properties of $v^*$, 
    \item a final contradiction is derived using all these properties. 
\end{itemize}

\begin{claim}\label{claim:ABC:weight is bigger than gain of neighbors}
$f(v;G) > \sum_{w \in N(v)}\gamma(w)$ holds for every vertex $v \in V(G)$.  In particular, $f(v;G) > \gamma(w)$ holds for every edge $vw \in E(G)$. 
\end{claim}

\begin{claimproof}
Let $v \in V(G)$. 
The fact that $G$ and $(A, B, C)$ form a minimum counterexample implies that $f(G, A, B, C) > f(G', A', B', C')$, where $G'\coloneqq G- v$ and $(A', B', C')$ is the restriction of $(A, B, C)$ to $V(G')$.  
In particular, 
\[f(G, A, B, C) > f(G', A', B', C') = f(G, A, B, C) - f(v;G) + \sum_{w \in N(v)}\gamma(w).\]
Adding $f(v;G)-f(G, A, B, C)$ on both sides yields the required inequality.
\end{claimproof}

We use the notation $N[S]$ for the set of vertices at distance at most $1$ from a given set $S$ of vertices in $G$, and $N^i(S)$ for the set of vertices at distance exactly $i$ from $S$.

\begin{claim}\label{claim:ABC:ILF size is bouded by sum of weights}
Let $F$ be an induced linear forest in $G$ respecting $(A, B, C)$ with $\abs {V(F)}\geq 1$.  
Let $G' \coloneqq G-N[V(F)]$ and let $(A', B', C')$ be the restriction of $(A, B, C)$ to $V(G')$. 
Then
\begin{align*}
\abs {V(F)} &< f(G, A, B, C) - f(G', A', B', C') \\
&\leq \sum_{v \in N[V(F)]}f(v;G) - \sum_{w \in N^2(V(F))}\gamma(w) \\
&\leq \sum_{v \in N[V(F)]}f(v;G).
\end{align*}
\end{claim}

\begin{claimproof}
By the minimality of our counterexample, we know that $G'$ contains an induced linear forest $F'$ with at least $f(G', A', B', C')$ vertices. 
(Note that possibly $V(G') = \varnothing$, in which case $V(F') = \varnothing$ and $f(G', A', B', C')=0$.) 
Since $F\cup F'$ is an induced linear forest in $G$ respecting $(A, B, C)$, it follows that 
\[f(G, A, B, C) > \abs {V(F \cup F')} 
= \abs {V(F)} + \abs {V(F')} 
\geq \abs {V(F)} + f(G', A', B', C').\]
Subtracting $f(G', A', B', C')$ from both sides yields the first inequality of the claim. 

The second inequality of the claim follows from the fact that the three functions $f_A(d)$, $f_B(d)$, and $f_C(d)$ are non increasing and that $d_{G'}(v) \leq d_{G}(v)-1$ holds for every vertex $v\in N^2_G(V(F))$, while $d_{G'}(v) = d_{G}(v)$ for every $v\in V(G')-N^2_G(V(F))$. 

The last inequality follows from the fact that $\gamma(w) \geq 0$ holds for every vertex $w$ of $G$. 
\end{claimproof}

\begin{claim}\label{claim:ABC:(1-weight) is bigger than weight of neighbors}
Let $v \in V(G)$. Then $1-f(v;G) < \sum_{w \in N(v)}f(w;G)$, and thus in particular $\max_{w \in N(v)}f(w;G) > \frac {1-f(v;G)}{d(v)}$.
\end{claim}

\begin{claimproof}
Immediate by \Cref{claim:ABC:ILF size is bouded by sum of weights} with $F = G[\{v\}]$.
\end{claimproof}

\begin{claim}\label{claim:ABC:demoting and lowering degree costs at most 1/6}
Let $v \in V(G)$. Then decreasing $d(v)$ by one and then demoting $v$ lowers $f(v;G)$ by at most $\frac 16$.
\end{claim}

\begin{claimproof}
For $d \geq 1$, define $D_A(d) \coloneqq f_A(d)-f_B(d-1)$ 
and $D_B(d) \coloneqq f_B(d)-f_C(d-1)$. 
Proving the claim amounts to showing that $D_A(d) \leq \frac 16$ and $D_B(d) \leq \frac 16$ for every $d \geq 1$.  

We have $D_A(1) = D_A(2) = -\frac 16 < 0$, $D_A(3) = \frac 16$ and $D_A(d) = \frac {2(d-2)}{3d(d+1)} < \frac 2{3d} \leq \frac 16$ for $d \geq 4$. 
Thus, $D_A(d) \leq \frac 16$ for every $d \geq 1$. 

We have $D_B(1) = -\frac 16 < 0$, $D_B(2)=D_B(3) = \frac 16$ and $D_B(d) = \frac {2(d-1)}{3d(d+1)} < \frac 2{3d} \leq \frac 16$ for $d \geq 4$. 
Thus, $D_B(d) \leq \frac 16$ for every $d \geq 1$. 
\end{claimproof}

Let $\delta(G)$ denote the minimum degree of a vertex in $G$. 

\begin{claim}\label{claim:no leaf in G}
$\delta(G) \geq 2$.
\end{claim}

\begin{claimproof}
We already know that $\delta(G)\geq 1$ since $G$ is connected and has at least three vertices. 

Let $vw \in E(G)$. 
Using \Cref{claim:ABC:weight is bigger than gain of neighbors} and the fact that $d(v) \geq 1$, we deduce that 
$\gamma(w) < f(v;G) \leq \frac 56$. 
This implies in particular that $C_1 = \varnothing$. (This is because it can be checked from the definition of $\gamma(w)$ that $\gamma(w) \leq \frac 56$, with equality if and only if $w\in C_1$.) 

Now, suppose that there is some vertex $v$ with degree $1$ in $G$.  Then $v \in A \cup B$ (as discussed above), and thus $f(v; G) = \frac 56$.
Let $w$ be the unique neighbor of $v$. \Cref{claim:ABC:weight is bigger than gain of neighbors} ensures that $f(w; G) > \gamma(v) = \frac 16$, thus $w \notin C$.
Let $G' \coloneqq G-v$, and let $(A', B', C')$ be the partition of $V(G')$ obtained by restricting $(A, B, C)$ to $V(G')$ and demoting $w$. 
It follows from \Cref{claim:ABC:demoting and lowering degree costs at most 1/6} that 
\begin{align*}
f(G', A', B', C') &= f(G, A, B, C) - f(v;G) - (f(w;G) - f(w;G')) \\ 
&\geq f(G, A, B, C) - f(v;G) - \frac 16 \\
&\geq f(G, A, B, C)-1.
\end{align*}
By the minimality of our counterexample, we know that $G'$ contains an induced linear forest $F'$ respecting $(A', B', C')$ with at least $f(G', A', B', C')$ vertices. 
Then, $F\coloneqq G[V(F') \cup \{v\}]$ is an induced linear forest of $G$ respecting $(A, B, C)$ and having at least $\abs {V(F')} + 1\geq f(G', A', B', C') +1 \geq f(G, A, B, C)$ vertices, a contradiction.

Therefore, $\delta(G) \geq 2$, as claimed.
\end{claimproof}

\begin{claim}\label{claim:ABC:no B_2 and no C_2}
$B_2 = C_2 = \varnothing$. 
\end{claim}

\begin{claimproof}
Suppose that $v$ is a degree-$2$ vertex of $G$. 
Denote by $u$ and $w$ its two neighbors, in such a way that $f(u;G) \geq f(w;G)$.

If $v \in B_2$, then by \Cref{claim:ABC:weight is bigger than gain of neighbors}, we know in particular $f(u;G) \geq f(w;G) > \gamma(v) = \frac 12$. 
Since $\delta(G) \geq 2$, it follows that $u, w \in A_2$, and thus $\gamma(u)=\gamma(w) = \frac 16$. 
But then \Cref{claim:ABC:weight is bigger than gain of neighbors} also states that $\frac 13 = f(v;G) > \gamma(u)+\gamma(w) = \frac 13$, 
which is a contradiction.

If $v \in C_2$, then by \Cref{claim:ABC:(1-weight) is bigger than weight of neighbors} we know that $f(u;G)+f(w;G) > 1-f(v;G) = \frac 56$, which implies that $f(u;G) > \frac 5{12}$, which implies in turn that $u \in A_2 \cup A_3$ and $\gamma(u) = \frac 16$. However, it follows then from \Cref{claim:ABC:weight is bigger than gain of neighbors} that $\frac 16 =  f(v;G) > \gamma(u)  = \frac 16$, a contradiction.

We deduce that $v$ is neither in $B_2$ nor in $C_2$ (and thus is in $A_2$), as claimed. 
\end{claimproof}

Given two sets $X$ and $Y$ of vertices of $G$, we call  an edge of $G$ with one end in $X$ and the other end in $Y$ an {\em $X$--$Y$ edge}. 

\begin{claim}\label{claim:ABC:no A_2-B_3 edge}
$G$ contains no $A_2$--$B_3$ edge.
\end{claim}

\begin{claimproof}
Suppose that $v$ is a vertex in $B_3$. 
It follows from \Cref{claim:ABC:weight is bigger than gain of neighbors} that $\abs {N(v) \cap (A_2 \cup A_3)} \leq 1$. Thus, denoting $x, y, z$ the three neighbors of $v$ in such a way that $f(x;G) \geq f(y;G) \geq f(z;G)$, it follows that at most one of them is in $A_2$, and if there is one, it must be $x$. 
Arguing by contradiction, let us suppose that $x \in A_2$. 
Then $f(y;G) \leq \frac 25$ and $f(z;G) \leq \frac 25$. 

Let $G' \coloneqq G-\{v, y, z\}$.
By the minimality of our counterexample, there is an induced linear forest $F'$ in $G'$ of size at least $f(G', A', B', C')$ respecting the partition $(A', B', C')$ obtained by restricting $(A, B, C)$ to $V(G')$. 
Now, observe that $F \coloneqq G[V(F') \cup \{v\}]$ is an induced linear forest in $G$ respecting $(A, B, C)$ whose number of vertices is at least
\begin{align*}
f(G', A', B', C') + 1 
&\geq f(G, A, B, C) + 1 - f(v;G) - f(y;G) - f(z;G) + \gamma(x) \\
&\geq f(G, A, B, C) + \frac 1{30} \\
&\geq f(G, A, B, C),
\end{align*} 
which contradicts the fact that $G$ and $(A, B, C)$ form a counterexample.
\end{claimproof}

\begin{claim}\label{claim:ABC:loss in neighbourhood of B_3 is at most 1/6}
Let $v \in B_3$, and let $x, y$ be two distinct neighbors of $v$ such that $f(x; G) \geq f(y; G)$.
Then $\lambda(x) \leq \frac 1{10}$ and $\lambda(y) \leq \frac 1{15}$.
In particular $\lambda(x)+\lambda(y) \leq \frac 16$.
\end{claim}

\begin{claimproof}
\Cref{claim:ABC:no A_2-B_3 edge} implies that $x\notin A_2$, thus $d_G(x)\geq 3$, and $\lambda(x) \leq \frac  1{10}$. 
The same is true for vertex $y$. 

By \Cref{claim:ABC:weight is bigger than gain of neighbors}, 
$\gamma(x) + \gamma(y) < f(v;G) = \frac 13$. 
Thus, $x$ and $y$ cannot both be in $A_3$ (as they would then each have a gain of $\frac 16$), and it follows that $y\notin A_3$ since $f(x;G) \geq f(y;G)$. 
This implies in turn that $\lambda(y) \leq \frac {1}{15}$, and thus 
\[
\lambda(x) + \lambda(y) \leq \frac 1{10}+\frac 1{15} = \frac 16. 
\qedhere
\]
\end{claimproof}

\begin{claim}\label{claim:ABC:min weight in B_3 neighborhod is > 1/6}
$\min_{w \in N(v)}f(w;G) > \frac 16$ for every vertex $v \in B_3$. In particular, $G$ contains no $B_3$--$C$ edge.
\end{claim}

\begin{claimproof} 
Suppose that $v$ is a vertex in $B_3$ and denote by $x, y, z$ its three neighbors in such a way that $f(x;G) \geq f(y;G) \geq f(z;G)$. 
By \Cref{claim:ABC:loss in neighbourhood of B_3 is at most 1/6}, we know that $\lambda(x) + \lambda(y) \leq \frac 16$. 


Now, let $G' \coloneqq G-z + xy$ (note that the edge $xy$ might already be in $G-z$, in which case it is not added). 
Let $(A', B', C')$ be the partition of $V(G')$ obtained by restricting $(A, B, C)$ to $V(G')$ and promoting $v$ (thus, $v\in A'_2$). 
Observe that every induced linear forest $F'$ in $G'$ respecting $(A', B', C')$ is also an induced linear forest in $G$ respecting $(A, B, C)$, since the edge $xy$ ensures that $F'$ cannot contain all three vertices $v, x, y$. 
By the minimality of our counterexample, $F'$ can be chosen so that it contains at least $f(G', A', B', C')$ vertices. 
Thus, we must have $f(G', A', B', C') < f(G, A, B, C)$. 
Observe that $d_G(x)-1 \leq d_{G'}(x) \leq d_G(x)+1$ (depending on whether the edges $xy$ and $xz$ are in $G$ or not), and thus in particular
\[f(x; G) - f(x; G') \leq \lambda(x).\]
The same holds for vertex $y$. 
Hence, letting $Z\coloneqq N_G(z) - \{v, x, y\}$, we obtain
\begin{align*}
0 &< f(G, A, B, C)-f(G', A', B', C') \\
&= f(z; G) + \underbrace {f(v; G)-f(v; G')}_{= \frac 13 - \frac 23 = -\frac 13} + \underbrace {f(x; G) - f(x; G')}_{\leq \lambda(x)} + \underbrace {f(y; G) - f(y; G')}_{\leq \lambda(y)}  - \underbrace {\sum_{w \in Z}\gamma(w)}_{\geq 0} \\
&\leq f(z;G) - \frac 13 + \lambda(x)+\lambda(y) - 0  \\
&\leq f(z;G) - \frac 13 + \frac 16 = f(z;G) - \frac 16.
\end{align*} 
Therefore, $f(z;G) > \frac 16$.
\end{claimproof}

\begin{claim}\label{claim:ABC:at most a single B_3 in a B_3's neighborhood}
If $v \in B_3$, then $\abs {N(v) \cap B_3} \leq 1$.
\end{claim}

\begin{claimproof}
Suppose that $v$ is a vertex in $B_3$ and denote by $x, y, z$ its three neighbors. 
Arguing by contradiction, suppose that at least two of them are in $B_3$, say $x, y\in B_3$. 
From \Cref{claim:ABC:min weight in B_3 neighborhod is > 1/6}, we know that $z$ cannot be in $C$. Furthermore, by \Cref{claim:ABC:no B_2 and no C_2} and \Cref{claim:ABC:no A_2-B_3 edge}, we know that $d(z) \geq 3$. 

Let $G' \coloneqq G-v +xy +yz + xz$. 
(Note that some of the edges $xy,yz,xz$ might already be in $G$.) 
Let $(A', B', C')$ be the partition of $V(G')$ obtained by restricting $(A, B, C)$ to $V(G')$ and demoting $x, y$, and $z$. 
Since $|V(G')| < |V(G)|$, the graph $G'$ contains an induced linear forest $F'$ respecting $(A', B', C')$ with at least $f(G', A', B', C')$ vertices. 
Now, let $F \coloneqq G[V(F') \cup \{v\}]$. 
Observe that $F$ is an induced linear forest in $G$ respecting $(A, B, C)$, since $F'$ contains at most one of $x, y, z$. 
Observe also that $d_{G'}(x) \leq d_{G}(x) +1$, and thus $f(x;G) - f(x;G') \leq f_B(3) - f_C(4) = \frac 15$. 
The same holds for vertex $y$. 
Moreover,  
\[f_A(d) - f_B(d+1) = \frac {2(d+4)}{3(d+1)(d+2)}\qquad\text{ and }\qquad f_B(d)-f_C(d+1) = \frac {2(d+3)}{3(d+1)(d+2)}\]
for $d\geq 3$. 
Hence, since $d_{G'}(z) \leq d_{G}(z) +1$, we deduce that 
$f(z;G) - f(z;G') \leq \frac {2(3+4)}{3(3+1)(3+2)} = \frac 7{30}$. 
Therefore, 
\begin{align*}
\abs {V(F)} &\geq f(G', A', B', C') + 1 \\
&= f(G, A, B, C) - f(v;G) - \sum_{w\in N(v)}(f(w;G) - f(w;G')) + 1 \\
&\geq f(G, A, B, C) - f(v;G) - 2 \cdot \frac 15 - \frac 7{30} + 1  \\
&= f(G, A, B, C) + \frac 1{30} \\
&\geq f(G, A, B, C),
\end{align*}
contradicting the fact that $G$ and $(A, B, C)$ form a counterexample.
\end{claimproof}

\begin{claim}\label{claim:ABC:no A_3 adjacent to two B_3's}
If $v \in A_3$, then $\abs {N(v) \cap B_3} \leq 1$.
\end{claim}

\begin{claimproof}
Suppose that $v$ is a vertex in $A_3$ and denote by $x, y, z$ its three neighbors. 
Arguing by contradiction, suppose that at least two of them are in $B_3$, say $x, y\in B_3$. 
Note that by \Cref{claim:ABC:weight is bigger than gain of neighbors}, we know that $f(z; G) > \gamma(v) = \frac 16$, thus $z \notin C$, and hence $z$ can be demoted.
Let $G' \coloneqq G-v + xy$, and let $(A', B', C')$ be the partition of $V(G')$ obtained by restricting $(A, B, C)$ to $V(G')$ and demoting $x,y,z$
(thus, $x,y$ are both in $C'_2$ or both in $C'_3$, depending on whether the edge $xy$ exists in $G$ or not). 
Since $|V(G')| < |V(G)|$, the graph $G'$ contains an induced linear forest $F'$ respecting $(A', B', C')$ with at least $f(G', A', B', C')$ vertices. 
Let $F \coloneqq G[V(F') \cup \{v\}]$. 
Observe that $F$ is an induced linear forest in $G$ respecting $(A, B, C)$. 
Indeed, $d_F(w) \leq d_{F'}(w)+1$ for every $w\in N(v)$, and 
$d_F(v) \leq 2$ since $F'$ contains at most one of $x, y$ (because of the edge $xy$), and furthermore $v$ is not in a cycle in $F$ since $x,y\in C'$. 
But then
\begin{align*}
    \abs {V(F)} &= \abs {V(F')} + 1 \\
    &\geq f(G', A', B', C') + 1 \\
    &\geq f(G, A, B, C) - f(v;G) - 3 \cdot \frac 16 + 1 \\
    &= f(G, A, B, C),
\end{align*}
where the second inequality is obtained by \Cref{claim:ABC:demoting and lowering degree costs at most 1/6}, which ensures that
$f(z; G)-f(z; G') \leq \frac 16$, and by the fact that $f_C(2)=f_C(3) = \frac 16$, hence for $w \in \{x, y\}$, we know that
$f(w; G)-f(w; G') = \frac 13 - \frac 16 = \frac 16$.
But this contradicts the fact that $G$ and $(A, B, C)$ form a counterexample.
\end{claimproof}

\begin{claim}\label{claim:ABC:at most one common neighbor between B_3's}
If $u, v$ are two distinct vertices in $B_3$, then $\abs {N(u) \cap N(v)} \leq 1$.
\end{claim}

\begin{claimproof}
Arguing by contradiction, suppose that $u, v$ are two distinct vertices in $B_3$ with two common neighbors $x, y$. 
We know from \Cref{claim:ABC:no A_2-B_3 edge} and \Cref{claim:ABC:no A_3 adjacent to two B_3's} that neither $x$ nor $y$ can be in $A_2 \cup A_3$. 
In particular $f(x;G) < \frac 12$ and $f(y;G) < \frac 12$. 
Let $G' \coloneqq G-\{x, y\}$, and let $F'$ be an induced linear forest in $G'$ respecting the restriction of $(A,B,C)$ to $V(G')$ with at least $f(G', A', B', C')$ vertices.  
Obviously, $F'$ is an induced linear forest in $G$ respecting  $(A,B,C)$ as well.  
Then, 
\begin{align*}
\abs {V(F')} &\geq f(G', A', B', C') \\
&\geq f(G, A, B, C)-f(x;G)-f(y;G)+2\left(f_B(1)-f_B(3)\right) \\
&> f(G, A, B, C) - 2 \cdot \frac 12 + 2\left(\frac 56 - \frac 13\right) \\
&= f(G, A, B, C),
\end{align*}
which is a contradiction.
\end{claimproof}

\begin{claim}\label{claim:ABC:no C_3--C_3 edge}
$G$ does not contain any $C_3$--$C_3$ edge.
\end{claim}

\begin{claimproof}
Arguing by contradiction, suppose that $vw$ is an edge of $G$ with $v, w \in C_3$. 
Denote by $x, y$ and $s, t$ the other two neighbors of respectively $v$ and $w$, in such a way that $f(x;G) \geq f(y;G)$ and $f(s;G) \geq f(t;G)$. 
By \Cref{claim:ABC:(1-weight) is bigger than weight of neighbors} we know that $f(x;G)+f(y;G) > 1-f(v;G)-f(w;G)=\frac 23$, and in particular $f(x;G) > \frac 13$. Furthermore, \Cref{claim:ABC:weight is bigger than gain of neighbors} implies that $\gamma(x) < \frac 16$, and thus $x \in A_4$.
From \Cref{claim:ABC:weight is bigger than gain of neighbors}, we know that $\gamma(y) < f(v; G)-\gamma(x) = \frac 16 - \frac 1{10} = \frac 1{15}$. 
Finally, since $f(y;G) > \frac 23-f(x;G) = \frac 4{15}$,  we deduce that either $y \in A_6$ or $y\in B_3$. However, \Cref{claim:ABC:min weight in B_3 neighborhod is > 1/6} implies that $y\notin B_3$, hence $y \in A_6$. 
A symmetric argument shows that $s \in A_4$ and $t \in A_6$.

Let $S\coloneqq N^2(w) \cap \{x, y\}$. 
Let $G'\coloneqq G- N[w]$ and let  $(A', B', C')$ be the restriction of $(A, B, C)$ to $V(G')$. 
Since $\gamma(z) \geq \frac 1{21}$ holds for every vertex $z \in S$, 
applying \Cref{claim:ABC:ILF size is bouded by sum of weights} with $F \coloneqq G[\{w\}]$ we obtain
\begin{align*}
1 &< f(G, A, B, C) - f(G', A', B', C') \\
&= f(v;G)+f(w;G)+f(s;G)+f(t;G) + \sum_{z \in N^2(w)} (f(z;G) - f(z;G')) \\
&\leq f(v;G)+f(w;G)+f(s;G)+f(t;G) + \sum_{z \in S} (f(z;G) - f(z;G')) \\
&\leq f(v;G)+f(w;G)+f(s;G)+f(t;G) - \sum_{z \in S}\gamma(z) \\
&\leq \frac {107}{105} - \frac 1{21}\abs {S},
\end{align*}
implying that $\abs {S} < \frac 25$, and thus $\abs {S} = 0$. 
Hence, $x, y \notin N^2(w)$, and therefore $x, y \in N(w)$, implying that $x=s$ and $y=t$. 

Let $G''\coloneqq G-\{v, w\}$, and let  $(A'', B'', C'')$ be the restriction of $(A, B, C)$ to $V(G'')$. 
By the minimality of our counterexample, we have $f(G, A, B, C) > f(G'', A'', B'', C'')$. 
However, 
\begin{align*}
f(G'', A'', B'', C'') 
&= f(G, A, B, C) - 2f_C(3) + \left(f_A(2) - f_A(4)\right) + \left(f_A(4) - f_A(6)\right)  \\
&= f(G, A, B, C) + \frac 1{21} \\
&\geq f(G, A, B, C),
\end{align*}
which is a contradiction.
\end{claimproof}

Let us show that some vertex of $G$ has strictly positive gain. 
Observe that if $\gamma(v)=0$ holds for a vertex $v \in V(G)$ then $v$ must belong to one of the three sets $B_3, C_2, C_3$. 
Moreover, $C_2 = \varnothing$ by \Cref{claim:ABC:no B_2 and no C_2}. 
Thus, if all vertices of $G$ have a gain of zero, then from the connectedness of $G$ and \Cref{claim:ABC:min weight in B_3 neighborhod is > 1/6} we deduce that either all vertices are in $B_3$, or they are in all $C_3$. 
However, \Cref{claim:ABC:at most a single B_3 in a B_3's neighborhood} rules out the former case, while \Cref{claim:ABC:no C_3--C_3 edge} rules out the latter. 
Therefore, some vertex of $G$  has strictly positive gain. 

For the rest of the proof, we fix some vertex $v^*$ with $\gamma(v^*) > 0$  and minimizing $\gamma(v^*)$ among all such vertices. 
Furthermore, in case $\gamma(v^*) = \frac 16$ and $A_3 \neq \varnothing$, we choose $v^*$ so that $v^* \in A_3$.   

Observe that $d(v^*) \geq 2$ (since $\delta(G) \geq 2$ by \Cref{claim:no leaf in G}) and $d(v^*) \geq 4$ in case $v^*\in B \cup C$.

\begin{claim}
\label{claim:ABC:properties of v^*}
The vertex $v^*$ satisfies the following three properties: 
\begin{enumerate}
    \item \label{item not in A2} $v^* \notin A_2$,    
    \item \label{item formula weight v} $f(v^*;G) = d(v^*)\gamma(v^*)$, and
    \item \label{item neighbor in B3 or C3} $N(v^*) \cap (B_3 \cup C_3) \neq \varnothing$. 
\end{enumerate}
\end{claim}
\begin{claimproof}
Let us first show that \eqref{item not in A2} implies \eqref{item formula weight v}. 
We already know that $d(v^*) \geq 4$ in case $v^* \in B \cup C$. 
If $v^* \notin A_2$, then it can be checked that $f(v^*;G) = d(v^*)\gamma(v^*)$ holds in all possible cases for $v^*$. 
Thus,  \eqref{item not in A2} implies \eqref{item formula weight v}.  

From \eqref{item formula weight v} and \Cref{claim:ABC:weight is bigger than gain of neighbors}, we obtain that 
$d(v^*)\gamma(v^*) = f(v^*;G) > \sum_{w \in N(v^*)}\gamma(w)$, implying that 
$\gamma(w) < \gamma(v^*)$ holds for some $w \in N(v^*)$, and thus $\gamma(w)=0$ by our choice of $v^*$, implying in turn that $w\in B_3 \cup C_3$. Thus,  \eqref{item formula weight v} implies \eqref{item neighbor in B3 or C3}. 

Hence, it only remains to prove \eqref{item not in A2}. 
Arguing by contradiction, suppose that $v^* \in A_2$. 
Then, it follows from our choice of $v^*$ that $V(G)=A_2 \cup B_3 \cup C_3$. 

If $B_3 \neq \varnothing$, then by \Cref{claim:ABC:no A_2-B_3 edge} and \Cref{claim:ABC:min weight in B_3 neighborhod is > 1/6} every vertex in $B_3$ has all its neighbors in $B_3$. 
However, this contradicts \Cref{claim:ABC:at most a single B_3 in a B_3's neighborhood}. 
Thus, $B_3 = \varnothing$. 

If $C_3 \neq \varnothing$, then every vertex in $C_3$ has all its neighbors in $A_2$ by \Cref{claim:ABC:no C_3--C_3 edge}. 
In particular, there is an edge $xy\in E(G)$ with $x\in C_3$ and $y\in A_2$. 
But then, $f(x;G) > \gamma(y)$ by \Cref{claim:ABC:weight is bigger than gain of neighbors}, which is not possible since $f(x;G) = \gamma(y) = \frac 16$. 
Thus, $C_3 = \varnothing$. 

We conclude that $V(G)=A_2$, and thus that $G$ is a cycle. 
Then $F\coloneqq G - v^*$ is an induced linear forest in $G$ respecting $(A, B, C)$ and with $|V(G)|-1 \geq \frac 23 |V(G)|=f(G, A, B, C)$ vertices, contradicting the fact $G$ and $(A, B, C)$ form a counterexample. 
Therefore,  $v^* \notin A_2$, and \eqref{item not in A2} holds. 
\end{claimproof}

\begin{claim}
\label{claim:ABC:deg v^* at least 4}
$d(v^*) \geq 4$, and thus 
$f(v^*;G) \leq \frac {2}{5}$ and
$\gamma(v^*) \leq \frac {1}{10}$.     
\end{claim}

\begin{claimproof}
If $d(v^*) \geq 4$, then it is easily checked that $f(v^*;G) \leq \frac {2}{5}$ (with equality only if $v^* \in A_4$), and thus $\gamma(v^*) \leq \frac {1}{10}$ since $\gamma(v^*) = f(v^*;G)/d(v^*)$ by \Cref{claim:ABC:properties of v^*}.

Let us show that $d(v^*) \geq 4$. 
Arguing by contradiction, suppose this is not the case. 
Then it follows from \Cref{claim:ABC:properties of v^*} that $v^*\in A_3$. 
Furthermore, by our choice of $v^*$, every vertex of $G$ with nonzero gain is in $A_2 \cup A_3$, and it follows that $V(G) = A_2 \cup A_3 \cup B_3 \cup C_3$. By \Cref{claim:ABC:weight is bigger than gain of neighbors}, we know that $f(w;G) > \gamma(v^*) = \frac 16$ holds for every vertex $w\in N(v^*)$. 
Since $f_C(3) = \frac 16$, it follows that $N(v^*) \cap C_3 = \varnothing$.   
\Cref{claim:ABC:properties of v^*} implies then that there is some vertex $w \in N(v^*) \cap B_3$. 
By \Cref{claim:ABC:min weight in B_3 neighborhod is > 1/6}, we know that $N(w) \subseteq B_3 \cup A_2 \cup A_3$. 
\Cref{claim:ABC:at most a single B_3 in a B_3's neighborhood} implies that $w$ has a neighbor $x$ that is distinct from $v^*$ and not in $B_3$. Thus, $x \in A_2 \cup A_3$. 
\Cref{claim:ABC:no A_2-B_3 edge} implies then that $x \in A_3$. But then $\gamma(x)+\gamma(v^*)=f(w;G)$, which contradicts \Cref{claim:ABC:weight is bigger than gain of neighbors}.
\end{claimproof}

\begin{claim}\label{claim:ABC:If C3 nonempty}
If $C_3 \neq \varnothing$ then $v^* \notin A_4 \cup A_5 \cup B_4$, and in particular, $f(v^*; G) \leq \frac 27$ and $\gamma(v^*) \leq \frac 1{21}$.
\end{claim}
\begin{claimproof}
Suppose $v$ is a vertex in $C_3$. 
Then it follows from \Cref{claim:ABC:min weight in B_3 neighborhod is > 1/6} and \Cref{claim:ABC:no C_3--C_3 edge} that $\gamma(w) > 0$ for every neighbor $w$ of $v$. Using \Cref{claim:ABC:weight is bigger than gain of neighbors}, we deduce then that 
\[
\frac{1}{6}=f(v;G) > \sum_{w \in N(v)}\gamma(w) \geq 3\gamma(v^*), 
\]
and thus $\gamma(v^*)< \frac 1{18}$.  
This implies that $v^* \notin A_4 \cup A_5 \cup B_4$, and in particular $f(v^*; G) \leq \frac 27$ and $\gamma(v^*) \leq \frac 1{21}$.
\end{claimproof}

\begin{claim}
\label{claim:ABC:v^* has at most one neighbor in C3}
$\abs {N(v^*) \cap C_3} \leq 1$.
\end{claim}

\begin{claimproof}
Arguing by contradiction, suppose that $u, w \in N(v^*) \cap C_3$ with $u \neq w$. Then $uw \not \in E(G)$ by \Cref{claim:ABC:no C_3--C_3 edge}. 
Denote by $x, y$ the two neighbors of $u$ distinct from $v^*$, and by $s, t$  the two neighbors of $w$ distinct from $v^*$. 
We may assume without loss of generality that $f(x;G) \geq f(y;G)$, $f(s;G) \geq f(t;G)$ and $f(x;G) \geq f(s;G)$. 
\Cref{claim:ABC:weight is bigger than gain of neighbors} implies that, for every vertex $z \in N(u) \cup N(w)$, we have $\gamma(z) < \frac 16$, and thus $z \notin A_2 \cup A_3$, which implies in turn that $f(z;G) \leq \frac 25$. 
Since $uw \not \in E(G)$, using \Cref{claim:ABC:ILF size is bouded by sum of weights} with $F=G[\{u, w\}]$ we obtain
\[2 < f(u;G) + f(w;G) + \sum_{z \in N(u) \cup N(w)}f(z;G) \leq 2 \cdot \frac 16 + \abs {N(u) \cup N(v)}\frac 25,\]
and thus $\abs {N(u) \cup N(w)} \geq \ceil {\frac {25}6} = 5$. 
Hence, $v^*$ is the only common neighbor of $u$ and $w$, and $x, y, s, t$ are all distinct. 
Furthermore, \Cref{claim:ABC:min weight in B_3 neighborhod is > 1/6} and \Cref{claim:ABC:no C_3--C_3 edge} imply that $\{x, y, s, t\} \cap (B_3 \cup C_3) = \varnothing$. 
In particular, $x, y, s, t$ all have non-zero gain.
Since $C_3 \neq \varnothing$, by \Cref{claim:ABC:If C3 nonempty} we know that $v^* \notin A_4$.
Using \Cref{claim:ABC:ILF size is bouded by sum of weights} with $F=G[\{u,w\}]$ again, we obtain that
\begin{align*}
2 &< 2 \cdot \frac 16 + f(x;G) + f(y;G) + f(s;G) + f(t;G) + f(v^*;G) \\
&\leq \frac 13 + 5\max\{f(x;G), f(y;G), f(s;G), f(t;G), f(v^*;G)\} \\
&= \frac 13 + 5\max\{f(x;G), f(v^*; G)\}.
\end{align*}
(To see the last equality, recall that $f(x;G) \geq f(y;G)$ and $f(x;G) \geq f(s;G) \geq f(t;G)$). 
We deduce that $\max\{f(x;G), f(v^*; G)\} > \frac 13$, thus $\{x, v^*\} \cap A_4 \neq \varnothing$, but since $v^* \notin A_4$ we must have $x \in A_4$ and thus $f(x;G) = \frac 25$ and $\gamma(x) = \frac 1{10}$. 
Applying \Cref{claim:ABC:weight is bigger than gain of neighbors} on $u$ again, we obtain 
\[\frac 16 = f(u;G) > \gamma(x) + \gamma(y) + \gamma(v^*) \geq \gamma(x) + 2\gamma(v^*) = \frac 1{10} + 2\gamma(v^*).\]
Thus, $\gamma(v^*) < \frac 12(\frac 16-\frac 1{10}) = \frac 1{30}$. 
It follows that $v^*$ does not belong to any of the sets $A_4, A_5, A_6, A_7, B_4, B_5, C_4$, and hence $f(v^*;G) \leq \frac 29$. 
Since \Cref{claim:ABC:weight is bigger than gain of neighbors} implies in particular that $\frac 16 = f(u;G) > \gamma(x)+\gamma(y)$, it follows that $\gamma(y) < \frac 16-\gamma(x) = \frac 16-\frac 1{10}=\frac 1{15}$, from which we deduce that $y$ is not in $A_4, A_5, B_4$, and hence $f(y;G) \leq \frac 27$. 
Now, observe that
\begin{align*}
    \frac 29 
    &\geq f(v^*;G) \\
    &> 2 - 2 \cdot \frac 16 - f(x;G) - f(y;G) - (f(s;G)+f(t;G)) \\
    &\geq \frac 53 - \frac 25 - \frac 27 - \left(f(s;G)+f(t;G)\right) \\
    &= \frac {103}{105} - \left(f(s;G)+f(t;G)\right),
\end{align*}
which implies
\[f(s;G)+f(t;G) > \frac {103}{105} - \frac 29 = \frac {239}{315} > \frac 23.\]
In particular, $f(s;G) > \frac 13$, and thus $s \in A_4$ (since $s\notin A_2\cup A_3$), and $f(s;G) = \frac 25$. 
Furthermore, $f(t;G) > \frac {239}{315}-\frac 25 = \frac {113}{315} > \frac 13$ implying $t \in A_4$.
However, applying \Cref{claim:ABC:weight is bigger than gain of neighbors} on $w$ we obtain
\[\frac 16 = f(w;G) > \gamma(s) + \gamma(t) = 2 \cdot \frac 1{10} = \frac 15,\]
which is a contradiction.
\end{claimproof}

\begin{claim}
\label{claim:ABC: no A4-B3-B3 triangle with v^* in A4}
If $v^* \in A_4$ and $v^*$ is in a triangle $u, v^*, w$ in $G$, then at least one of $u, w$ is not in $B_3$. 
\end{claim}

\begin{claimproof}
Arguing by contradiction, suppose that $v^*\in A_4$ and that $u, v^*, w$ is a triangle in $G$ with $u, w\in B_3$. 
Let $x$ be the neighbor of $u$ distinct from $v^*$ and $w$, and let $y$ be the neighbor of $w$ distinct from $v^*$ and $u$.  
By \Cref{claim:ABC:at most one common neighbor between B_3's}, we know that $x \neq y$.  
\Cref{claim:ABC:no A_2-B_3 edge} implies that $x, y \notin A_2$. 
It follows that $f(x;G) \leq \frac 12$ and $f(y;G) \leq \frac 12$. 

Since $f(v^*;G) = \frac {2}{5}$, using \Cref{claim:ABC:ILF size is bouded by sum of weights} with $F=G[\{u, w\}]$, we obtain that 
\begin{align*}
2 &< f(u;G)+f(w;G)+f(x;G)+f(y;G)+f(v^*;G) \\
&= \frac 23 + f(x;G)+f(y;G)+f(v^*;G) \\
&= \frac {16}{15} +f(x;G)+f(y;G),
\end{align*}
and thus $f(x;G)+f(y;G) > \frac {14}{15}$.
Since $x, y \notin A_2$, we thus must have $x, y\in A_3$.  

Reapplying \Cref{claim:ABC:ILF size is bouded by sum of weights} with $F=G[\{u, w\}]$, and letting $Z \coloneqq N^2(V(F))$, we obtain
$2 < 2 + \frac 1{15} - \sum_{z \in Z}\gamma(z)$, 
that is, \[\sum_{z \in Z}\gamma(z) < \frac 1{15}.\] 
Since every vertex with non-zero gain has gain at least $\gamma(v^*)=\frac 1{10}$, we deduce that $Z$ is a (possibly empty) subset of $B_3 \cup C_3$. 
Since $x \in A_3$, \Cref{claim:ABC:weight is bigger than gain of neighbors} and \Cref{claim:ABC:no A_3 adjacent to two B_3's} imply that $(N(x) \setminus \{u\}) \cap (B_3 \cup C_3) = \varnothing$. 
Since $N(x) \subseteq N[V(F)] \cup Z$, and since $v^*, x, y$ are the only vertices of $N[V(F)] \cup Z$ not in $B_3 \cup C_3$, it follows that $N(x) = \{u, v^*, y\}$. By symmetry, we deduce that $N(y) = \{w, v^*, x\}$. But then $V(G) = \{u, v^*, w, x, y\}$, and
\begin{align*}
f(G, A, B, C) &= f(u;G)+f(w;G)+f(v^*;G)+f(x;G)+f(y;G) \\
&= 2f_B(3)+f_A(4)+2f_A(3) \\
&= 2 + \frac 1{15} \\
&< 3.
\end{align*}
However, $G[\{x, y, w\}]$ is an induced linear forest of $G$ with three vertices respecting $(A, B, C)$, contradicting the fact that $G$ and $(A, B, C)$ form a counterexample.
\end{claimproof}

\begin{claim}
\label{claim:ABC:v^* has at most one neighbor in B3}
$\abs {N(v^*) \cap B_3} \leq 1$.
\end{claim}

\begin{claimproof}
Arguing by contradiction, suppose that $u, w \in N(v^*) \cap B_3$ with $u \neq w$.

First, suppose that $uw \in E(G)$. 
Let $x$ be the neighbor of $u$ distinct from $v^*$ and $w$, and let $y$ be the neighbor of $w$ distinct from $v^*$ and $u$.  
By \Cref{claim:ABC:at most one common neighbor between B_3's}, we know that $x \neq y$.  
Also, $f(v^*;G) \leq \frac 13$ since $d(v^*)\geq 4$ (by \Cref{claim:ABC:deg v^* at least 4}) and $v^* \notin A_4$ (by \Cref{claim:ABC: no A4-B3-B3 triangle with v^* in A4}).  
Using \Cref{claim:ABC:ILF size is bouded by sum of weights} with $F=G[\{u, w\}]$ we then obtain that 
\begin{align*}
2 &< f(u;G)+f(w;G)+f(x;G)+f(y;G)+f(v^*;G) \\
&= \frac 23 + f(x;G)+f(y;G)+f(v^*;G) \\
&\leq 1 +f(x;G)+f(y;G),
\end{align*}
and thus $f(x;G)+f(y;G) > 1$. 
However, this implies that at least one of $x, y$ is in $A_2$, which contradicts \Cref{claim:ABC:no A_2-B_3 edge}. 
Therefore, we conclude that $uw \not \in E(G)$.

Let $x, y$ be the two neighbors of $u$ distinct from $v^*$, and let $s, t$ be  the two neighbors of $w$ distinct from $v^*$. 
We may assume without loss of generality that $f(x;G) \geq f(y;G)$ and $f(s;G) \geq f(t;G)$, and also $\lambda(x)+\lambda(y) \geq \lambda(s)+\lambda(t)$. 
\Cref{claim:ABC:at most one common neighbor between B_3's} implies that $x, y, s, t$ are all distinct. 
\Cref{claim:ABC:no A_2-B_3 edge} implies that $x, y \notin A_2$. 
It follows that $f(x;G) \leq \frac 12$ and $\gamma(x) \leq \frac 16$, and these two inequalities are strict in case $x \notin A_3$. 
The same holds for vertex $y$. 
By \Cref{claim:ABC:weight is bigger than gain of neighbors}, we know that $\gamma(x)+\gamma(y) < f(u;G)= \frac 13$, thus we cannot have both $x$ and $y$ in $A_3$, and hence $y \notin A_3$ (since $f(y;G) \leq f(x;G)$). 
It follows that $\lambda(x)+\lambda(y) \leq \frac 1{10}+\frac 1{15} = \frac 16$, and also  $\lambda(s)+\lambda(t) \leq \lambda(x)+\lambda(y) \leq \frac 16$. 

Let $G' \coloneqq G-v^*+xy+st$. 
(Note that the edges $xy$ and $st$ might already be in $G$.) 
Let $(A', B', C')$ be obtained from the restriction of $(A, B, C)$ to $V(G')$ by promoting $u$ and $w$ (thus $u, w \in A'_2$). 
Observe that every induced linear forest $F'$ in $G'$ respecting  $(A', B', C')$ is also an induced linear forest in $G$ respecting  $(A, B, C)$. 
(Indeed, the triangle $u, x, y$ in $G'$ ensures that $F'$ misses at least one of these three vertices, and same for the triangle $w, s, t$ in $G'$.) 
Since the lemma holds true for $G'$ and  $(A', B', C')$, it follows that
\[f(G, A, B, C) > f(G', A', B', C') \geq f(G, A, B, C) - f(v^*;G) - \lambda(x)-\lambda(y)-\lambda(s)-\lambda(t) + \frac 23,\] 
where the rightmost inequality holds with equality in case $xy, st \notin E(G)$. 
That is, 
\[\frac 23 - f(v^*;G) < \lambda(x)+\lambda(y)+\lambda(s)+\lambda(t).\]
Using that $\lambda(s)+\lambda(t) \leq \lambda(x)+\lambda(y) \leq \frac 16$, we deduce that $f(v^*;G) >  \frac 13$, 
and thus $v^* \in A_4$ and $f(v^*;G)= \frac 25$.  
Also, 
\[\frac 4{15} = \frac 23 - f(v^*;G) < \lambda(x)+\lambda(y)+\lambda(s)+\lambda(t) \leq 2\left(\lambda(x)+\lambda(y)\right),\]
implying that $\lambda(x)+\lambda(y) > \frac 2{15}$. 
In particular, $\lambda(z) > \frac 1{15}$ holds for some $z \in \{x, y\}$, and it can then be checked that $z\in A_3$ is the only possibility for $z$ since $x,y \notin A_2$. 
Since $y\notin A_3$, we must have $z=x$.  
Since $v^* \in A_4$, we know from \Cref{claim:ABC:weight is bigger than gain of neighbors} that $\frac 13 = f(u;G) > \gamma(v^*) + \gamma(x) + \gamma(y) = 
\frac 1{10} + \frac 16 + \gamma(y)$, and thus $\gamma(y) < \frac 13 - \frac 16 - \frac 1{10} = \frac 1{15} < \gamma(v^*)=\frac 1{10}$. 
By our choice of $v^*$, this implies that $\gamma(y)=0$, and thus $y \in B_3 \cup C_3$. Since \Cref{claim:ABC:min weight in B_3 neighborhod is > 1/6} ensures $y \not \in C$, we deduce that $y \in B_3$. 
\Cref{claim:ABC:no A_3 adjacent to two B_3's} implies in turn that $xy \notin E(G)$.  
Moreover, $v^*y \notin E(G)$ by \Cref{claim:ABC: no A4-B3-B3 triangle with v^* in A4}. 
It follows that $y$ has exactly two neighbors outside of $\{x, u, v^*\}$.
Denote these two neighbors by $q$ and $r$ in such a way that $f(q) \geq f(r)$. (Note that $\{q, r\} \cap \{s, t\}$ is not necessarily empty.) 

Since $u\in B_3$, \Cref{claim:ABC:at most a single B_3 in a B_3's neighborhood} applied on $y$ implies that $q, r \notin B_3$. 
Also, $q, r \notin C$ by \Cref{claim:ABC:min weight in B_3 neighborhod is > 1/6}. 
Thus, $\gamma(q) > 0$ and $\gamma(r) > 0$, and hence $\gamma(q) \geq \gamma(v^*) =\frac 1{10}$ and $\gamma(r) \geq \gamma(v^*) =\frac 1{10}$ by our choice of $v^*$. 
This implies in turn that $q, r\notin B$, and hence $q, r\in A$. 
Since $q, r \notin A_2$ by \Cref{claim:ABC:no A_2-B_3 edge}, we must have $q, r \in A_3 \cup A_4$. 
In particular, $\lambda(q)+\lambda(r) \leq 2 \cdot \frac 1{10}$. 

Now, let $G'' \coloneqq G-u+qr$ (note that the edge $qr$ may already exist in $G$) and let $(A'', B'', C'')$ be the partition $(A, B, C)$ restricted to $V(G'')$ where the vertex $y$ is promoted. 
Since $|V(G'')| < |V(G)|$, there is an induced linear forest $F''$ in $G''$ respecting $(A'', B'', C'')$ with at least $f(G'',A'', B'', C'')$ vertices.  
Observe that $F''$ is also an induced linear forest in $G$ respecting $(A, B, C)$, thanks to the triangle $y, q, r$ in $G''$.  
It follows that 
\begin{align*}
|V(F'')| &\geq f(G'',A'', B'', C'') \\
&= f(G, A, B, C) - f(u; G) + \gamma(v^*) + \gamma(x) + (f(y; G'') - f(y; G)) - (\lambda(q)+\lambda(r)) \\
&\geq f(G, A, B, C)-\frac 13 + \frac 1{10} + \frac 16 + \left(\frac 23-\frac 13\right) - 2 \cdot \frac 1{10} \\ 
&\geq f(G, A, B, C),
\end{align*}
contradicting the fact that $G$ and $(A, B, C)$ form a counterexample. 
This concludes the proof of the claim. 
\end{claimproof}

\begin{claim}\label{claim:ABC:v^* has no neighbor in B3}
$N(v^*) \cap B_3 = \varnothing$.
\end{claim}

\begin{claimproof}
Arguing by contradiction, suppose that $N(v^*) \cap B_3$ is non empty. By \Cref{claim:ABC:v^* has at most one neighbor in B3}, there exists a unique vertex $w$ in  $N(v^*) \cap B_3$. 
Let $x, y$ denote the neighbors of $w$ distinct from $v^*$.
Similarly to the proof of \Cref{claim:ABC:min weight in B_3 neighborhod is > 1/6}, let $G' := G - v^* + xy$ and define the partition $(A', B', C')$ as the restriction of $(A, B, C)$ to $V(G')$ where $w$ is moved to $A$ (hence $w \in A'_2$). 
Observe that every induced linear forest in $G'$ respecting $(A', B', C')$ is an induced linear forest in $G$ respecting $(A, B, C)$; indeed, if $w$ is included in such a forest, then at most one of $x,y$ is included as well, because of the edge $xy$. 
Hence, $f(G, A, B, C) > f(G', A', B', C')$ (by minimality of the counterexample), and it follows that
\begin{align*}
    0 &> f(G', A', B', C') - f(G, A, B, C) \\
    &= (f_A(2) - f_B(3)) - f(v^*; G) + (f(x; G') - f(x; G)) + (f(y; G') - f(y; G)) + \sum_{u \in N(v^*) \setminus \{w, x, y\}}\hspace{-.75cm}\gamma(u) \\
    &= \frac 13 - f(v^*; G) + (f(x; G') - f(x; G)) + (f(y; G') - f(y; G)) + \sum_{u \in N(v^*) \setminus \{w, x, y\}}\gamma(u).
\end{align*}
Define $U \coloneqq \{u \in N(v^*) \setminus \{w, x, y\} \st \gamma(u) > 0\}$. 
It follows that
\begin{equation}
\label{eq:claim19_beginning}
\frac 25 \geq f(v^*; G) > \frac 13 + (f(x; G') - f(x; G)) + (f(y; G') - f(y; G)) + \abs U \cdot \gamma(v^*),
\end{equation}
where the first inequality comes from \Cref{claim:ABC:deg v^* at least 4}.

If $f(x; G')-f(x; G) \geq 0$ and $f(y; G')-f(y; G) \geq 0$, we deduce that
\[\frac 25 \geq f(v^*; G) > \frac 13 + \abs U \cdot \gamma(v^*) \geq \frac 13,\]
thus $v^* \in A_4$, and hence $\gamma(v^*) = \frac{1}{10}$, and it follows that $\abs U < (\frac 25 - \frac 13)/\gamma(v^*) = \frac 23$, and therefore $\abs U=0$. 
Recall that $v^*$ has exactly one neighbor in $B_3$, namely $w$, at most one neighbor in $C_3$ (by \cref{claim:ABC:v^* has at most one neighbor in C3}), and all its remaining neighbors $u$ satisfy  $\gamma(u) > 0$. 
Since $\abs U=0$, it follows that $N(v^*) = \{w, x, y, s\}$ for some $s\in C_3$. 
But then \Cref{claim:ABC:If C3 nonempty} implies that $v^* \notin A_4$, which is a contradiction.

Hence, at least one of $f(x; G')-f(x; G)$ and $f(y; G')-f(y; G)$ must be negative.
Without loss of generality, suppose that $f(x; G')-f(x; G) < 0$.
This means that $d_{G'}(x) = d_G(x)+1$, hence $xy \notin E(G)$ and $v^*x \notin E(G)$.

We also know that $f(x; G')-f(x; G) = -\lambda(x)$ and $f(y; G')-f(y; G) \in \{0, -\lambda(y)\}$ depending on whether $v^*y \in E(G)$ or not. 
We claim that $C_3 = \varnothing$.
Indeed, it follows from the rightmost inequality in \eqref{eq:claim19_beginning} that 
\[f(v^*; G) > \frac 13 - \lambda(x) + \left(f(x; G') - f(x; G)\right) + \abs U \cdot \gamma(v^*),\]
hence
\begin{equation}\label{eq:claim 19}
\left(d_G(v^*)-\abs U\right)\gamma(v^*) > \frac 13 - \left(\lambda(x) + f(y; G) - f(y; G')\right) \geq \frac 13 - \left(\lambda(x)+\lambda(y)\right).
\end{equation}
Furthermore, $\abs U \geq d_G(v^*) - 3$, since $N(v^*)-U$ consists of $w$, potentially $y$ but not $x$, and at most one other vertex $s$ with $\gamma(s)=0$ and $s\in C_3$ (by \cref{claim:ABC:v^* has at most one neighbor in C3}), 
implying $\abs{N(v^*)-U} \leq 3$. 
Also, by \Cref{claim:ABC:loss in neighbourhood of B_3 is at most 1/6}, we know that $\lambda(x) \leq \max\{\lambda(x), \lambda(y)\} \leq \frac 1{10}$ and $\lambda(x) + \lambda(y) \leq \frac 16$.
Hence
\[3\gamma(v^*) > \frac 13 - \frac 16 = \frac 16.\]
Therefore, $\gamma(v^*) > \frac 1{18}$, and by the contrapositive of \Cref{claim:ABC:If C3 nonempty}, we deduce that $C_3 = \varnothing$.
In particular, $d_G(v^*) - \abs U \in \{1, 2\}$, depending on whether $v^*y \in E(G)$ or not. 
We distinguish these two cases. 

{\bf Case~1: $v^*y \in E(G)$}. Then, since $\lambda(x) \leq \frac 1{10}$, 
$\gamma(v^*) \leq \frac 1{10}$ (by \Cref{claim:ABC:deg v^* at least 4}), 
$d_G(y) = d_{G'}(y)$ 
and $\abs U = d_G(v^*)-2$, by the leftmost inequality of \eqref{eq:claim 19}, we know that
\[\frac 15 \geq 2\gamma(v^*) > \frac 13 - \lambda(x) \geq \frac 7{30} > \frac 15,\]
which is a contradiction. 

{\bf Case~2: $v^*y \notin E(G)$}. Then, since $\lambda(x)+\lambda(y) \leq \frac 16$ and $\abs U = d_G(v^*)-1$, by \eqref{eq:claim 19}, we know that
\[\frac 1{10} \geq \gamma(v^*) > \frac 13 - \left(\lambda(x)+\lambda(y)\right) \geq \frac 16,\]
which is a contradiction as well.
    
This concludes the proof of the claim.
\end{claimproof}

\Cref{claim:ABC:properties of v^*}, \Cref{claim:ABC:v^* has at most one neighbor in C3}, and \Cref{claim:ABC:v^* has no neighbor in B3} together imply that $v^*$ has no neighbor in $B_3$ and exactly one neighbor in $C_3$, let us call it $w$. 
Let $x, y$ be the two neighbors of $w$ that are distinct from $v^*$ in  such  a way that $f(x; G) \geq f(y; G)$, and let $Z \coloneqq N(v^*) \setminus N[w]$. 
Observe that for every $z\in Z$, we have  $\gamma(z) > 0$, and thus $\gamma(v^*) \leq \gamma(z)$. 
Using \Cref{claim:ABC:ILF size is bouded by sum of weights} with $F=G[\{w\}]$, we obtain
\begin{align*}
1 &< f(w;G) + f(v^*;G) + f(x;G) + f(y;G) - \sum_{z \in Z}\gamma(z) \\
&\leq \frac 16 + f(v^*;G) + f(x;G) + f(y;G) - \abs Z\gamma(v^*).
\end{align*}
Since $C_3 \neq \varnothing$, \Cref{claim:ABC:If C3 nonempty} implies that $\gamma(v^*)\leq \frac 1{21}$. 
By \Cref{claim:ABC:properties of v^*}, we know that $f(v^*; G) = d(v^*)\gamma(v^*)$. 
Since $1 \leq d(v^*)-\abs Z \leq 3$, it follows that
\begin{align*}
\frac 56 &< d(v^*)\gamma(v^*) + f(x; G) + f(y; G) - (d(v^*)-3)\gamma(v^*) \\
&= f(x; G) + f(y; G) + 3\gamma(v^*) \\
&\leq f(x; G) + f(y; G) + 3 \cdot \frac 1{21},
\end{align*}
and thus
\[f(x; G) = \max\{f(x; G), f(y; G)\} \geq \frac 12\left(f(x; G) + f(y; G)\right) > \frac 12\left(\frac 56 - \frac 17\right) = \frac {29}{84} > \frac 13.\]
Since $\gamma(x) < \frac 16$ (by \Cref{claim:ABC:weight is bigger than gain of neighbors}), we deduce that $x \in A_4$ and thus $f(x; G) = \frac 25$ and $\gamma(x) = \frac 1{10}$.
Then, 
\[f(y; G) > \frac {29}{42} - f(x; G) = \frac {61}{210} > \frac 27,\]
and since $\gamma(y) < \frac 16$ (by \Cref{claim:ABC:weight is bigger than gain of neighbors}), it follows that  $y \in A_4 \cup A_5$. But now $\gamma(x)+\gamma(y) \geq \frac 1{10} + \frac 1{15} = \frac 16$, which contradicts \Cref{claim:ABC:weight is bigger than gain of neighbors}.

This final contradiction shows that $G$ and $(A, B, C)$ do not form a  counterexample. This concludes the proof of \Cref{lemma:A-B-C}.
\end{proof}

%% file: forests_of_stars.tex
\section{forests of stars}
\label{sec:star_forests}

As discussed in the introduction, a special kind of forests of caterpillars---or equivalently, graphs of pathwidth at most $1$---are forests of stars, where every connected component is a star. 
(Let us recall that an isolated vertex is considered to be a star in this paper.) 
Forests of stars correspond to graphs of treedepth at most $2$. 
The aim of this section is to study the impact of restricting caterpillars to be stars. 
The lower bound in \Cref{cor:caterpillar} is no longer true with this extra requirement, and in fact there are infinitely many extremal lower bounds for finding induced forests of stars in graphs as a function of their degree sequence. 
\Cref{thm:characterization of ELBs for forests of stars} gives a characterization of all these lower bounds, which we restate here for convenience.  

\thmforestsofstars*

The proof of \Cref{thm:characterization of ELBs for forests of stars} relies on the following lemma. 

\begin{lemma}[AB lemma for forests of stars]\label{lemma:A-B for forests of stars}
Let $G$ be a graph and let $(A, B)$ be a partition of its vertex set (where parts are possibly empty). 
Then $G$ contains an induced forest of stars $F$ satisfying
\begin{itemize}
    \item every edge $vw \in E(F)$ with $w\in B$ satisfies $v\in A$ and $d_F(v) = 1$, and    
    \item $\abs {V(F)} \geq \sum_{v \in A}f_A(d_G(v)) + \sum_{v \in B}f_B(d_G(v))$     
\end{itemize}
where
\[f_A(d) \coloneqq \begin{cases}1 &\text{ if } d=0\\\frac 56 &\text{ if } d=1\\\frac 35 &\text{ if } d=2\\\frac 2{d+1} &\text{ if } d\geq 3\end{cases}\qquad\qquad\text{ and }\qquad\qquad f_B(d) \coloneqq \frac 1{d+1}.\]
\end{lemma}

Before proving \Cref{lemma:A-B for forests of stars}, let us show that it implies \Cref{thm:characterization of ELBs for forests of stars}. 
In the subsequent proofs, an induced forest of stars $F$ in a graph $G$ is said to {\em respect} a partition $(A, B)$ of $V(G)$ if $F$ satisfies the first condition in \Cref{lemma:A-B for forests of stars} w.r.t.\ the partition $(A, B)$, namely, that every edge $vw \in E(F)$ with $w\in B$ satisfies $v\in A$ and $d_F(v) = 1$.  

\begin{proof}[Proof of \Cref{thm:characterization of ELBs for forests of stars} assuming \Cref{lemma:A-B for forests of stars}]
Let $\varepsilon \in \R$ with $0 \leq \varepsilon \leq \frac 16$. 
We first show that $f_\varepsilon$ is a lower bound for $\alpha_{\SS}$. 

Let $G$ be a graph. 
We need to show that $\alpha_{\SS}(G) \geq \sum_{v \in V(G)}f_\varepsilon(d_G(v))$. 
Clearly, it is enough to prove this in the case where $G$ is connected, so let us assume that $G$ is connected. 
Also, we may assume that $|V(G)| \geq 3$, since the inequality is easily seen to hold in case $G$ is isomorphic to $K_1$ or $K_2$. 
This implies in particular that no two leaves of $G$ are adjacent. 

Let $L \coloneqq \{v \in V(G) \text{ s.t. } d_G(v)=1\}$ be the set of leaves of 
$G$ (which is thus an independent set). 
Let $G' \coloneqq G - L$.  
Let $(A', B')$ be the partition of $V(G')$ defined as follows
\begin{align*}
    A' &\coloneqq \left\{v \in V(G) \setminus L \text{ s.t. } N_G(v) \cap L = \varnothing\right\},\\
    B' &\coloneqq \left\{v \in V(G) \setminus L \text{ s.t. } N_G(v) \cap L \neq \varnothing\right\}.
\end{align*}

By \Cref{lemma:A-B for forests of stars}, there exists  an induced forest of stars $F'$ respecting $(A', B')$ in $G'$ with at least $\sum_{v \in A'}f_{A}(d_{G'}(v)) + \sum_{v \in B'}f_{B}(d_{G'}(v))$ vertices. 
Let $F \coloneqq G[V(F') \cup L]$. 
We claim that $F$ is an induced forest of stars in $G$. 
It is clear that $F$ is a forest, since it is obtained from $F'$ by (possibly) adding some leaves.   
Now, suppose that some connected component of $F$ is not a star. 
Then, one can find an edge $vw$ in that component where both $v$ and $w$ have degree at least $2$ in $F$. 
These two vertices (and the edge $vw$) must be in $F'$ as well, since all vertices in $F - V(F')$ are leaves of $F$. 
At least one of $v,w$ is a leaf in $F'$ (since every component of $F'$ is a star), say $w$ is a leaf in $F'$. 
Since $w$ is no longer a leaf in $F$, it follows that $w\in B'$. 
By the conditions of \Cref{lemma:A-B for forests of stars}, this implies in turn that $v\in A'$ and $v$ is a leaf of $F'$.  
But then, $v$ has no leaf neighbor in $G$, and hence $v$ is a leaf of $F$ as well, a contradiction. 
Therefore, $F$ is forest of stars, as claimed.  

Next, we will show that $F$ has at least $\sum_{v \in V(G)}f_\varepsilon(d_G(v))$ vertices. 
First,  observe that 
\[f_{A}(d_{G'}(v)) = f_{A}(d_G(v)) \geq f_\varepsilon(d_G(v))\]
for every vertex $v \in A'$, since $v$ has degree at least $2$ in $G$.  

For $v \in B'$, let $\ell(v) \coloneqq \abs {N_G(v) \cap L}$, and let 
\begin{align*}
D(v) \coloneqq& \left(f_{B}(d_{G'}(v)) + \ell(v)\right) - \left(f_\varepsilon(d_G(v)) + \sum_{w \in N_G(v) \cap L}f_\varepsilon(d_G(w))\right) \\
=& f_{B}(d_{G'}(v)) + \ell(v)\varepsilon - f_\varepsilon(d_G(v)).
\end{align*}
Observe that $D(v) > 0$ in case $d_{G'}(v) = 0$.  
If $d_{G'}(v) \geq 1$, since $\ell(v) \geq 1$ and $d_G(v) \geq d_{G'}(v)+1$, we know that $d_G(v) \geq 2$, and in particular $f_\varepsilon(d_G(v)) \leq \frac 1{d_G(v)}+\varepsilon$.
Therefore, 
\[D(v) \geq \frac 1{d_{G'}(v)+1} + \varepsilon - \frac 1{d_G(v)} - \varepsilon \geq \frac 1{d_{G'}(v)+1} - \frac 1{d_{G'}(v)+1} = 0.\]

Thus, $D(v) \geq 0$ holds for every vertex $v \in B'$. 
It follows that 
\begin{align*}
    \abs {V(F)} &=  \abs {V(F')} + \abs L \\
    &\geq \sum_{v \in A'}f_{A}(d_{G'}(v)) + \sum_{v \in B'}f_{B}(d_{G'}(v)) + \abs L \\
    &\geq \sum_{v \in A'}f_\varepsilon(d_G(v)) + \sum_{v \in B'}f_{B}(d_{G'}(v)) + \abs L \\ 
    &= \sum_{v \in A'}f_\varepsilon(d_G(v)) + \sum_{v \in B'}\left(f_\varepsilon(d_G(v)) + D(v)\right) + \sum_{v \in L} (1-\varepsilon) \\     
    &\geq \sum_{v \in V(G)}f_\varepsilon(d_G(v)), 
\end{align*}
as desired. This concludes the proof that $f_\varepsilon$ is a lower bound for $\alpha_{\SS}$. 

Now, it remains to show that (i) every lower bound for $\alpha_{\SS}$ is dominated by $f_\varepsilon$ for some $\varepsilon$ satisfying $0 \leq \varepsilon \leq \frac 16$, and (ii) that these bounds are all extremal.
We will proceed as in the proof of \Cref{thm:characterization of ELBs for max-degree-k caterpillars}. 
Let us first show (i). 
Let $\varphi : \mathbb N \to \mathbb R$ be a lower bound for $\alpha_{\SS}$, and define $\varepsilon \coloneqq 1-\varphi(1)$.
Since $\varphi$ is a lower bound for $\alpha_{\SS}$, we know that
\[
\varphi(0) \leq \alpha_{\SS}(K_1) = 1 = f_\varepsilon(0)
\]
and $5\varphi(2) \leq  \alpha_{\SS}(C_5)$, and thus
\[
\varphi(2) \leq  \frac 15\alpha_{\SS}(C_5) = \frac 35.
\]
Furthermore, $\varepsilon$ was chosen so that $\varphi(1) = f_\varepsilon(1)$, and $\varphi(d) \leq \frac{2}{d+1}$ holds for all $d\geq 3$, as witnessed by the complete graph $K_{d+1}$. 

Thus, it only remains to show that $\varphi(d) \leq \frac 1d + \varepsilon$ for all $d\geq 2$. 
To do so, let $n \geq 2$, and define the graph $K'_n$ as follows: 
\[V(K'_n) = \{(v, i) : 1 \leq v \leq n, i \in \{0, 1\}\}\] 
and there is an edge between $(v, i)$ and $(w, j)$ if and only if  $i=j=0$ and $v \neq w$, or $i=0$, $j=1$ and $v=w$.
Informally, $K'_n$ is the graph obtained by adding a single leaf to every vertex of the complete graph $K_n$.

We claim that $\alpha_{\SS}(K'_n) = n+1$.
One can observe that $\alpha_{\SS}(K'_n) \geq n+1$ since $\{(v, 1) : 1 \leq v \leq n\} \cup \{(1, 0)\}$ induces a forest of stars with exactly $n+1$ vertices.
To see that $\alpha_{\SS}(K'_n) \leq n+1$, observe that every induced forest of stars contains at most two vertices from $\{(v, 0) : 1 \leq v \leq n\}$; furthermore, if it contains two such vertices, then it contains at most $n-1$ vertices from $\{(v, 1) : 1 \leq v \leq n\}$, as otherwise it would contain an induced $P_4$. 

It follows that 
\[
n-n\varepsilon + n\varphi(n) = n\varphi(1) + n\varphi(n)  \leq \alpha_{\SS}(K'_n) = n+1
\]
and in particular
\[\varphi(n) \leq \frac 1n + \varepsilon, \]
as desired. 
We conclude that $\varphi \leq f_\varepsilon$.

Again, similarly to the argument in the proof of \Cref{thm:characterization of ELBs for max-degree-k caterpillars}, the functions $f_\varepsilon$ are pairwise incomparable for different values of $\varepsilon$, and thus must all be extremal.
\end{proof}

In order to prove \Cref{lemma:A-B for forests of stars}, we need the following easy lemma on cubic graphs. 
(A {\em cubic} graph is a graph where every vertex has degree $3$.)

\begin{lemma}\label{lemma:alpha_1 geq |G|/2 if G is cubic}
Every cubic graph $G$ admits a partition $(V_1, V_2)$ of $V(G)$ such that $\Delta(G[V_i]) \leq 1$ for $i \in \{1, 2\}$. 
In particular, $G$ contains an induced forest with at least 
$\frac 12\abs {V(G)}$ vertices, where every connected component is isomorphic to $K_1$ or $K_2$.
\end{lemma}

\begin{proof}
Given a partition $(V_1, V_2)$ of $V(G)$, we let $E(V_1, V_2)$ denote the set of $V_1$--$V_2$ edges in $G$. 
Observe that, if $(V_1, V_2)$ is a partition of $V(G)$, and if $v$ is a vertex with at least two neighbors in its own part, then moving $v$ to the other part increases $|E(V_1, V_2)|$ by at least $1$. 
It follows that if we let $(V_1, V_2)$ be a partition of $V(G)$ maximizing $|E(V_1, V_2)|$, then every vertex has at most one neighbor in its part, as desired. 
\end{proof}

We remark that \Cref{lemma:alpha_1 geq |G|/2 if G is cubic} is a special case of a theorem of Lov\'asz~\cite{lovasz1966on} about (non-proper) vertex colorings of graphs where each color class induces a graph whose maximum degree satisfies a prescribed upper bound. 

Now, we prove \Cref{lemma:A-B for forests of stars}. 

\begin{proof}[Proof of \Cref{lemma:A-B for forests of stars}]
We use similar notations as in \Cref{lemma:A-B-C}: Given a vertex $v$ of a graph $G$, and a partition $(A, B)$ of $V(G)$, we let $f(v; G, A, B)$ be $f_A(d_G(v))$ or $f_B(d_G(v))$, depending on whether $v\in A$ or $v\in B$. 
We abbreviate $f(v; G, A, B)$ into $f(v; G)$ when the partition $(A, B)$ is clear from the context. 
Also, we let $f(G, A, B)\coloneqq \sum_{v\in V(G)} f(v; G, A, B)$, and for $i \geq 0$, we let $A_i$ denote the set of vertices in $A$ with degree at most $i$, and we define $B_i$ similarly w.r.t.\ $B$. 

To prove the lemma, we argue by contradiction. 
Suppose the lemma is false, and let $G$ and $(A, B)$ be a counterexample minimizing $\abs{V(G)}$. 
Clearly, $G$ is connected, and $\abs{V(G)} \geq 3$. 

Observe that every vertex $v$ of $G$ has degree at least $1$. 
As in the proof of \Cref{lemma:A-B-C}, we define the \emph{gain} $\gamma(v)$ of a vertex $v$ to be the increase of $f(v;G)$ when $d(v)$ is decreased by one, that is, 
\[\gamma(v) \coloneqq \begin{cases}f_A(d(v)-1)-f_A(d(v)) &\text{ if } v \in A\\f_B(d(v)-1)-f_B(d(v)) &\text{ if } v \in B.\end{cases}\]
Observe that, contrary to the proof of \Cref{lemma:A-B-C}, here $\gamma(v) > 0$ always holds (that is, a gain of zero is not possible).
If $v \in B$, then $\gamma(v) = \frac 1{d(v)(d(v)+1)}$, and if $v \in A$:
\[\gamma(v) = \begin{cases}\frac 16 &\text{ if } d(v)=1\\\frac 7{30} &\text{ if } d(v)=2\\\frac 1{10} &\text{ if } d(v) = 3\\\frac 2{d(v)(d(v)+1)} &\text{ if } d(v) \geq 4\end{cases}\]

The proof is split into a number of claims. Our first two claims mirror respectively \Cref{claim:ABC:weight is bigger than gain of neighbors} and \Cref{claim:ABC:ILF size is bouded by sum of weights} in the proof of \Cref{lemma:A-B-C}. 
The proofs are verbatim the same (up to replacing $(A, B, C)$ with $(A, B)$ and `linear forest' by  `forest of stars'), thus we do not repeat them here. 

\begin{claim}\label{claim:forest of stars characterization:weight bigger than sum of gains}
$f(v;G) > \sum_{w \in N(v)}\gamma(w)$ holds for every vertex $v \in V(G)$.  
\end{claim}

\begin{claim}\label{claim:forest of stars characterization:weight of ISF is bigger than sum of gains}
Let $F$ be an induced forest of stars in $G$ respecting $(A, B)$ with $\abs {V(F)}\geq 1$.  
Let $G' \coloneqq G-N[V(F)]$ and let $(A', B')$ be the restriction of $(A, B)$ to $V(G')$. 
Then
\[\abs {V(F)} < f(G, A, B) - f(G', A', B')
\leq \sum_{v \in N[V(F)]}f(v;G) - \sum_{w \in N^2(V(F))}\gamma(w) 
\leq \sum_{v \in N[V(F)]}f(v;G).\] 
\end{claim}

In the next claim we show that $G$ has minimum degree at least $2$. 
Thus, this claim parallels \Cref{claim:no leaf in G} in the proof of \Cref{lemma:A-B-C}. 
However, the proofs of the two claims are different, because each depends on the particular values of $f(v;G)$ for a vertex $v$ in $G$, which are not the same in the two lemmas. 

\begin{claim}\label{claim:forest of stars characterization:no leaf}
$\delta(G) \geq 2$.
\end{claim}

\begin{claimproof}
Since $G$ is connected with at least three vertices, we already know that  $\delta(G) \geq 1$. 
Arguing by contradiction, suppose that $v$ is a vertex of $G$ with degree $1$. 
Let $w$ be its neighbor. 
Note that $d(w) \geq 2$, since $G$ is not isomorphic to $K_2$. 

{\bf Case~1: $v\in A_1$.}
Let $G' \coloneqq G-v$ and let $(A', B')$ be obtained by  restricting $(A, B)$ to $V(G')$ and moving $w$ to $B'$ in case it is not already there.  
By the minimality of our counterexample, we know that $G'$ contains an induced forest of stars $F'$ respecting $(A', B')$ with at least $f(G', A', B')$ vertices. 
Observe that $F\coloneqq G[V(F') \cup \{v\}]$ is an induced forest of stars of $G$ respecting $(A, B)$ (thanks to the fact that $w\in B')$.  
Observe also that $f(w; G) - f(w; G') \leq \frac 16$ holds in all possible cases for the vertex $w$ (with equality if $w\in A_3$). 
It follows that 
\begin{align*}
\abs {V(F)} 
&= \abs {V(F')} + 1 \\
&\geq f(G', A', B') + 1 \\
& = f(G, A, B) - f(v; G) - \left(f(w; G) - f(w; G')\right)  + 1 \\
&\geq f(G, A, B) - f(v; G) -\frac 16 + 1 \\
&=f(G, A, B),
\end{align*}
a contradiction. 
Hence, $v\notin A_1$.

{\bf Case~2: $v\in B_1$.}
\Cref{claim:forest of stars characterization:weight bigger than sum of gains} implies that $f(w; G) > \gamma(v) = \frac 12$. 
It follows that $w \in A$ and $d(w) \leq 2$, and hence $w\in A_2$ since $d(w) \geq 2$.  
Let $u$ be the neighbor of $w$ distinct from $v$. 
Note that $f(u;G ) \leq \frac 56$ since $d(u) \geq 1$. 
Using \Cref{claim:forest of stars characterization:weight of ISF is bigger than sum of gains} with $F=G[\{v,w\}]$, we obtain
\[2 < f(v;G)+f(w;G)+f(u;G) 
\leq \frac 12 + \frac 35 + \frac 56 
= \frac {29}{15},\]
a contradiction. 
This concludes the proof. 
\end{claimproof}

\begin{claim}\label{claim:forest of stars characterization:A_2-A_3_B_2}
$A = A_2 \cup A_3$ and $B=B_2$. 
\end{claim}
\begin{claimproof}
Let $v \in V(G)$ be a vertex of minimum gain. By \Cref{claim:forest of stars characterization:weight bigger than sum of gains} and our choice of $v$,
\[f(v;G) > \sum_{w \in N(v)}\gamma(w) \geq d(v)\gamma(v),\]
that is, $\gamma(v) < \frac 1{d(v)}f(v;G)$. 
This implies that $v\notin B$ and $v\notin A_i$ for $i \geq 4$. 
Recalling that $\delta(G) \geq 2$ (by \Cref{claim:forest of stars characterization:no leaf}), it follows that $v\in A_2 \cup A_3$, and in particular, $\gamma(v) \geq \frac 1{10}$.   
By our choice of vertex $v$, this implies in turn that every vertex of $G$ has gain at least $\frac 1{10}$, and thus belongs to $A_2$, $A_3$, or $B_2$. 
\end{claimproof}

\begin{claim}\label{claim:forest of stars characterization:no B_2-B_2 edge}
There is no $B_2$--$B_2$ edge in $G$.
\end{claim}

\begin{claimproof}
Arguing by contradiction, suppose that $vw$ is an edge of $G$ with $v, w\in B_2$. 
Let $x$ be the neighbor of $v$ distinct from $w$. 
By \Cref{claim:forest of stars characterization:weight bigger than sum of gains}, 
\[\frac 13 = f(v;G) > \gamma(w) + \gamma(x) = \frac 16 + \gamma(x),\]
that is, $\gamma(x) < \frac 16$. 
In particular, $x$ cannot be in $A_2$ nor in $B_2$, and thus is in $A_3$ by \Cref{claim:forest of stars characterization:A_2-A_3_B_2}.

{\bf Case~1: $xw \in E(G)$.}
Let $z$ be the neighbor of $x$ distinct from $v$ and $w$. 
Note that $f(z;G) \leq \frac 35$ by \Cref{claim:forest of stars characterization:A_2-A_3_B_2}. 
Using \Cref{claim:forest of stars characterization:weight of ISF is bigger than sum of gains} with $F=G[\{v,x\}]$, we obtain
\[2 < f(v;G) + f(w;G) + f(x;G) + f(z;G) = \frac 76 + f(z;G) \leq \frac 76 + \frac 35 = \frac {53}{30} < 2,\]
a contradiction.

{\bf Case~2: $xw \notin E(G)$.}
Let $y, z$ denote the two neighbors of $x$ distinct from $v$. 
Using \Cref{claim:forest of stars characterization:weight of ISF is bigger than sum of gains} with $F=G[\{v\}]$, we obtain 
\[1 < f(v;G) + f(w;G) + f(x;G) - \gamma(y) - \gamma(z) \leq 2 \cdot \frac 13 + \frac 12 - 2 \cdot \frac 1{10} = \frac {29}{30} < 1,\]
a contradiction.

Thus, both cases lead to a contradiction, which concludes the proof. 
\end{claimproof}

\begin{claim}\label{claim:forest of stars characterization:B_2 is empty}
$B_2 = \varnothing$.
\end{claim}

\begin{claimproof}
Arguing by contradiction, suppose that $v$ is a vertex in $B_2$. 
Denote by $u, w$ the two neighbors of $v$ in such a way that $f(u;G) \geq f(w;G)$. 
Then, $u, w \in A_2 \cup A_3$ by \Cref{claim:forest of stars characterization:A_2-A_3_B_2} and \Cref{claim:forest of stars characterization:no B_2-B_2 edge}. 
Furthermore, using \Cref{claim:forest of stars characterization:weight of ISF is bigger than sum of gains}, we obtain that $\frac 13 = f(v;G) > \gamma(u)+\gamma(w)$,
and hence $u,w \in A_3$. 

Let $W \coloneqq (N(u) \cup N(w)) \setminus \{v\}$. 
It follows from \Cref{claim:forest of stars characterization:A_2-A_3_B_2} that $W \subseteq A_2 \cup A_3 \cup B_2$, and by
\Cref{claim:forest of stars characterization:weight bigger than sum of gains}, we know that $N(u) \cap A_2 = \varnothing = N(w) \cap A_2$,
therefore $W \subseteq A_3 \cup B_2$, hence every $x \in W$ satisfies $f(x; G) \leq \frac 12$ and $\gamma(x) \geq \frac 1{10}$.

{\bf Case~1: $uw \in E(G)$.}
Let $x$ be the neighbor of $u$ distinct from $v$ and $w$. 
Using \Cref{claim:forest of stars characterization:weight of ISF is bigger than sum of gains} with $F=G[\{u,v\}]$, we obtain 
\[2 < f(u;G)+f(v;G)+f(w;G)+f(x;G) = 2 \cdot \frac 12 + \frac 13 + f(x; G) \leq \frac {11}6 < 2,\]
a contradiction.

{\bf Case~2: $uw \notin E(G)$.} 
Using \Cref{claim:forest of stars characterization:weight of ISF is bigger than sum of gains} with $F=G[\{v\}]$, we obtain 
\[
1 < f(v;G)+f(u;G)+f(w;G) - \sum_{x \in W}\gamma(x)
\leq \frac{4}{3} - \frac {\abs W}{10}, 
\]
that is, $\abs W < \frac {10}3$. 
Using the same claim with $F=G[\{v,u,w\}]$, we obtain 
\[
3 < f(v;G)+f(u;G)+f(w;G)+\sum_{x \in W}f(x; G)
\leq \frac{4}{3} + \frac {\abs W}{2},  
\]
that is, $\abs W > \frac {10}3$, which is impossible.

Thus, both cases lead to a contradiction, as desired. 
\end{claimproof}

It follows from \Cref{claim:forest of stars characterization:A_2-A_3_B_2} and \Cref{claim:forest of stars characterization:B_2 is empty} that $V(G) = A_2 \cup A_3$. 

\begin{claim}\label{claim:forest of stars characterization:at most one A_2 in N(A_3)}
If $v \in A_3$, then $\abs {N(v) \cap A_2} \leq 1$.
\end{claim}

\begin{claimproof}
Let $v \in A_3$ and let $x, y, z$ be the three neighbors of $v$. 
We may assume that $f(x; G) \geq f(y; G) \geq f(z; G)$.
Recall that $A = A_2 \cup A_3$ and $B=\varnothing$, by \Cref{claim:forest of stars characterization:A_2-A_3_B_2} and \Cref{claim:forest of stars characterization:B_2 is empty}. 
If $x \notin A_2$, then $N(v) \subseteq A_3$, and we are done.  
If $x \in A_2$, let us show that $y, z \in A_3$.
By \Cref{claim:forest of stars characterization:weight bigger than sum of gains}, we know that
\[\frac 12 = f(v; G) > \gamma(x) + \gamma(y) + \gamma(z) = \frac 7{30} + \gamma(y)+\gamma(z) \geq \frac 7{30} + 2\gamma(z).\]
Thus, $\gamma(z) < \frac 2{15} < \frac 7{30}$, and $z \notin A_2$, hence $z \in A_3$. Finally, we also know that
\[\gamma(y) < f(v; G) - \gamma(x) - \gamma(z) = \frac 12 - \frac 7{30} - \frac 1{10} = \frac 16 < \frac 7{30},\]
Therefore, $y \notin A_2$, and it follows that $y \in A_3$. 
\end{claimproof}

\begin{claim}\label{claim:forest of stars:no P_3 of A_2's}
Every $3$-vertex path in $G$ contains at least one vertex in $A_3$. 
\end{claim}

\begin{claimproof}
Arguing by contradiction, suppose that $xyz$ is a $3$-vertex path in $G$ avoiding $A_3$. 
Then $x, y, z \in A_2$.  

If $xz \in E(G)$ then $G$ is isomorphic to $K_3$. 
Then, $f(G, A, B)= \frac{9}{5} \leq 2$ and $F\coloneqq G[\{x,y\}]$ is an induced forest of stars respecting $(A, B)$, showing that $G$ and $(A, B)$ do not form a counterexample, a contradiction. Hence, $xz \notin E(G)$. 

Let $v$ be the neighbor of $x$ distinct from $y$, and let $w$ be the neighbor of $z$ distinct from $y$. 
If $v = w$ then by \Cref{claim:forest of stars characterization:at most one A_2 in N(A_3)}, we know that $v \in A_2$, and thus $G$ is isomorphic to a cycle of length $4$. 
Then, $f(G, A, B)= \frac{12}{5} \leq 3$ and $F\coloneqq G[\{x,y,z\}]$ is an induced forest of stars respecting $(A, B)$, showing that $G$ and $(A, B)$ do not form a counterexample, a contradiction. 
Hence, $v \neq w$. 

Since $v, w \in A_2 \cup A_3$, using \Cref{claim:forest of stars characterization:weight of ISF is bigger than sum of gains} with $F=G[\{x, y, z\}]$ we obtain
\[3 < f(x; G) + f(y; G) + f(z; G) + f(v; G) + f(w; G) \leq 5 f_A(2) = 3,\]
a contradiction.
\end{claimproof}

\begin{claim}
\label{claim:forests_of_stars_no_A2-A2_edge}
There is no $A_2$--$A_2$ edge in $G$. 
\end{claim}

\begin{claimproof}
Arguing by contradiction, suppose that $vw$ is an $A_2$--$A_2$ edge in $G$. 
Let $x$ be the neighbor of $v$ distinct from $w$, and let $y$ be the neighbor of $w$ distinct from $v$. 

{\bf Case~1: $x=y$.}
By \Cref{claim:forest of stars characterization:at most one A_2 in N(A_3)}, we know that $x \in A_2$, but then $G$ is isomorphic to $K_3$, which we know is not possible (as discussed in the proof of \Cref{claim:forest of stars:no P_3 of A_2's}). 

{\bf Case~2: $x\neq y$.}
Applying \Cref{claim:forest of stars:no P_3 of A_2's} on the two paths $xvw$ and $vwy$, we deduce that $x, y \in A_3$. 
Let $z$ be a neighbor of $y$ outside $\{w, x\}$. 

First, suppose that $xy \in E(G)$. 
Using \Cref{claim:forest of stars characterization:weight of ISF is bigger than sum of gains} with $F=G[\{v, w, y\}]$, we obtain 
\[3 < f(v;G)+f(w;G)+f(x;G)+f(y;G)+f(z;G) = \frac {11}5 + f(z;G) \leq \frac {11}5 + \frac 35 < 3,\] 
a contradiction. Thus $xy \notin E(G)$. 

Now, let $F \coloneqq G[\{v, w\}]$. 
Then $N[V(F)] = \{v, w, x, y\}$, and $\abs {N^2(V(F))} \geq 2$ since $x$ and $y$ have degree $3$ and are not adjacent. 
But then, using \Cref{claim:forest of stars characterization:weight of ISF is bigger than sum of gains} on $F$, we obtain 
\[2 < \frac {11}5 - \frac 1{10}\abs {N^2(V(F))} \leq 2,\]
a contradiction. 
This concludes the proof of the claim. 
\end{claimproof}

\begin{claim}\label{claim:forest of stars:no A2 in a triangle}     
If $\{x,y,z\}$ induces a triangle in $G$, then $x, y, z \in A_3$, and $\abs {N_G(\{x, y, z\}) \cap A_2} \leq 1$.
\end{claim}

\begin{claimproof}
Suppose that $\{x,y,z\}$ induces a triangle in $G$. 
We may assume that $f(x; G) \geq f(y; G) \geq f(z; G)$. 

First, we prove that $x, y, z \in A_3$. 
Arguing by contradiction, suppose this is not the case. 
Then, $x \in A_2$, since $V(G) = A_2 \cup A_3$ and $f(x; G) \geq f(y; G) \geq f(z; G)$.  
By \Cref{claim:forests_of_stars_no_A2-A2_edge}, we know that $y, z \notin A_2$, and thus $y, z \in A_3$.
Let $G' \coloneqq G-x$ and let $(A', B')$ be the restriction of $(A, B)$ to $V(G')$ with $y, z$ moved to $B'$.
Observe that
\begin{align*}
f(G', A', B') &= f(G, A, B) - f(x; G) - f(y; G) - f(z; G) + f(y; G') + f(z; G') \\
&= f(G, A, B) - \frac 35 - 2\left(f_A(3)-f_B(2)\right) \\
&= f(G, A, B) - \frac 35 - 2 \cdot \frac 16 \\
&= f(G, A, B) - \frac {14}{15}.
\end{align*}
By the minimality of our counterexample, there exists an induced forest of stars $F'$ in $G'$ respecting $(A', B')$ and containing at least $f(G', A', B')$ vertices.
Let $F \coloneqq G[V(F') \cup \{x\}]$. 
Observe that $F$ is an induced forest of stars in $G$. 
Indeed, since $y, z \in B'$ and $yz \in E(G')$, they cannot both be in $V(F')$.
Furthermore, if one of them is in $V(F')$, say $y \in V(F')$, then $N_F(y) = N_{F'}(y) \cup \{x\} \subseteq A' \cup \{x\}$, and it follows that all vertices in $N_F(y)$ have degree 1 in $F$, as desired.
Note also that $F$ (trivially) respects $(A, B)$, since all vertices of $F$ are in $A$. 
The number of vertices in $F$ is then
\[
\abs {V(F)} = \abs {V(F')}+1
\geq f(G', A', B') + 1
= f(G, A, B) + \frac 1{15}
\geq f(G, A, B),\]
contradicting the fact that $G$ and $(A, B)$ form a counterexample. 
Therefore, $x, y, z \in A_3$, as claimed. 

Now, let us show that $\abs {N_G(\{x, y, z\}) \cap A_2} \leq 1$. 
Let $u$ (resp.\ $v$, $w$) be the neighbor of $x$ (resp.\ $y$, $z$) not in the triangle $xyz$.
Arguing by contradiction, suppose that $\abs {N_G(\{x, y, z\}) \cap A_2} \geq 2$. 
Without loss of generality, we may assume that $u\neq v$ and $u,v \in A_2$, thus $f(u; G) = f(v; G) = \frac 35$. 
Let $t$ be the neighbor of $v$ distinct from $y$.
By \Cref{claim:forests_of_stars_no_A2-A2_edge}, we know that $t \notin A_2$, and thus $t \in A_3$.
In particular, $t \neq u$. 
By the first part of this proof, we also know that $t \neq x$ and $t \neq z$, otherwise there would be a triangle containing a degree-2 vertex.

{\bf Case 1: $t = w$.}
Observe that $N_G[\{v, y, z\}] = \{x, y, z, v, w\}$, and $w \neq u$ (since $w=t \in A_3$ and $u\in A_2$). 
Applying \Cref{claim:forest of stars characterization:weight of ISF is bigger than sum of gains} on $F \coloneqq G[\{v, y, z\}]$, we obtain 
\[3 < \underbrace {f(x; G) + f(y; G) + f(z; G) + f(w; G)}_{=4 \cdot f_A(3) 
= 4 \cdot \frac 12} + \underbrace {f(v; G)}_{= f_A(2)} = 2 + \frac 35 < 3,\]
which is a contradiction.

{\bf Case 2: $t \neq w$.}
Let $S \coloneqq \{x, y, v\}$, let $W \coloneqq N(S) = \{z, u, t\}$ and let $Z \coloneqq N^2(S)$.
Applying \Cref{claim:forest of stars characterization:weight of ISF is bigger than sum of gains} on $F \coloneqq G[S]$, we obtain
\[3 < f(x; G)+f(y; G)+f(v;G) + f(z; G) + f(u; G)+f(t; G) - \frac 1{10}\abs Z = 4 \cdot \frac 12 + 2 \cdot \frac 35 - \frac 1{10}\abs Z.\]
In particular, this implies that $\abs Z \leq 1$.
But since $w \in Z$, we deduce that $Z = \{w\}$.
This implies however that $N(t) \setminus \{v\} \subseteq \{x, y, z, w\}$.
But we know that $t \neq u$, and hence $x \notin N(t)$, and also $y \notin N(t)$ by the first part of this proof.
It follows that $N(t) = \{v, z, w\}$, and in particular $z \in N(t)$, therefore $t = w$, which is a contradiction.

Both cases lead to a contradiction, which concludes the proof of the claim. 
\end{claimproof}

\begin{claim}\label{claim:forest of stars:4 distinct vertices at distance 2 from an A2}
If $v \in A_2$, then $\abs {N^2(v)} = 4$.
\end{claim}

\begin{claimproof}
Suppose that $v \in A_2$ and let $u, w$ be the two neighbors of $v$. 
By \Cref{claim:forest of stars:no A2 in a triangle}, we know that $uw \not \in E(G)$, and by \Cref{claim:forests_of_stars_no_A2-A2_edge}, we know that $u, w \in A_3$.
Let $W \coloneqq N^2(v) = N(\{u, w\}) \setminus \{v\}$.
We know that $W \subseteq A_3$ by \Cref{claim:forest of stars characterization:at most one A_2 in N(A_3)}.
Let $F \coloneqq G[\{v, u, w\}]$ (thus $W = N(V(F))$).
Applying \Cref{claim:forest of stars characterization:weight of ISF is bigger than sum of gains} on $F$, we obtain
\[3 < f(v; G) + f(u; G) + f(w; G) + \sum_{x \in W}f(x; G) = \frac 35 + 2 \cdot \frac 12 + \abs W\frac 12.\]
Thus $\abs W > 2\left(3 - \frac 35\right) - 2 = \frac {14}5$, and hence $\abs W \geq 3$. Furthermore, we know that $\abs W \leq 4$, and hence $\abs W \in \{3, 4\}$.

To prove the lemma, we must show that $\abs W = 4$. 
Arguing by contradiction, suppose not, thus $\abs W = 3$. 
Denote the three vertices in $W$ as $x, y, z$ in such a way that $N(u) = \{v, x, y\}$ and $N(w) = \{v, x, z\}$.
Let $Z \coloneqq N^2(V(F))$. Applying again \Cref{claim:forest of stars characterization:weight of ISF is bigger than sum of gains} on $F$, we obtain
\[3 < \frac 35 + \frac 52 - \sum_{t \in Z}\gamma(t) \leq \frac {31}{10} - \abs Z\frac 1{10}.\]
Thus $\abs Z < 1$, that is, $Z=\varnothing$, and hence $N[W] = W \cup \{u, w\}$. Since $w \notin N(y)$, we deduce that $N(y) = \{u, x, z\}$.
Similarly, since $u \notin N(z)$, we deduce that $N(z) = \{w,x,y\}$.
But then $N(x) = \{u, w, y, z\}$, which is a contradiction since $x$ has degree $3$ (and $u, w, y, z$ are all distinct). 
We conclude that $\abs W = 4$, as desired.
\end{claimproof}

\begin{claim}\label{claim:forest of stars:distance between A2's is >= 4}
If $v$ and $w$ are two distinct vertices in $A_2$, then $\dist_G(v, w) \geq 4$.
\end{claim}

\begin{claimproof}
Arguing by contradiction, suppose that there are two distinct vertices $v, w$ in $A_2$ such that $\dist_G(v, w) \leq 3$. 
\Cref{claim:forests_of_stars_no_A2-A2_edge} ensures that $\dist_G(v, w) > 1$, thus $\dist_G(v, w) \in \{2, 3\}$.

If $\dist_G(v, w) = 2$, then there is some vertex $u \in N(v) \cap N(w)$.
By \Cref{claim:forests_of_stars_no_A2-A2_edge}, we know that $u \notin A_2$, thus $u \in A_3$ since $V(G) = A_2 \cup A_3$.
But then $u$ contradicts \Cref{claim:forest of stars characterization:at most one A_2 in N(A_3)}.
It follows that $\dist_G(v, w) = 3$.

Consider a length-$3$ path $vxyw$ from $v$ to $w$ in $G$. 
Let $s$ be the neighbor of $v$ distinct from $x$ and let $t$ be that of $w$ distinct from $y$. 
Observe that $s \neq t$, since otherwise $vsw$ would be a length-$2$ path from $v$ to $w$, contradicting $\dist_G(v,w) = 3$. 
For the same reason, $s \neq y$ and $t \neq x$. 
Thus $x, y, s, t$ are all distinct.
Note also that $x, y, s, t \in A_3$ by \Cref{claim:forests_of_stars_no_A2-A2_edge}. 
By \Cref{claim:forest of stars:no A2 in a triangle}, we know that $sx \notin E(G)$ and $yt \notin E(G)$.
By \Cref{claim:forest of stars:4 distinct vertices at distance 2 from an A2}, we also know that $xt \notin E(G)$ and $sy \notin E(G)$.

{\bf Case 1: $st \in E(G)$.}
Let $x'$ (resp.\ $y', s', t'$) be the neighbor of $x$ (resp.\ $y, s, t$) not in the cycle $vxywts$.
We know that $x', y', s', t' \in A_3$ since no vertex in $A_3$ can have more than one neighbor in $A_2$ by \Cref{claim:forest of stars characterization:at most one A_2 in N(A_3)}. 
Let us show that these four vertices are all distinct.
By \Cref{claim:forest of stars:no A2 in a triangle}, we know that $x' \neq y'$ and $s' \neq t'$. 
By \Cref{claim:forest of stars:4 distinct vertices at distance 2 from an A2}, we also know that $x' \neq s'$ and $y' \neq t'$.
If $x' = t'$, then let $G' \coloneqq G - \{x, t\}$ and let $(A', B')$ be the restriction of $(A, B)$ to $V(G')$.
Note that $f(x'; G) = \frac 12$ and $f(x'; G') = \frac 56$.
By the minimality of our counterexample, we know that $G'$ admits an induced forest of stars $F'$ respecting $(A', B')$ whose number of vertices is at least
\begin{align*}
f(G', A', B') &= f(G, A, B) - f(x; G) - f(t; G) + \gamma(v)+\gamma(w) + f(x'; G')-f(x'; G) + \gamma(y)+\gamma(s) \\
&= f(G, A, B) - 2 \cdot \frac 12 + 2 \cdot \frac 7{30} + \frac 56 - \frac 12 + 2 \cdot \frac 1{10} \\
&= f(G, A, B).
\end{align*}
But $F'$ is also a induced forest of stars in $G$ respecting $(A, B)$, contradicting the fact that $G$ and $(A, B)$ form a counterexample.
We deduce that $x' \neq t'$, and by a similar argument, one can show that $y' \neq s'$.
It follows that $x', y', s', t'$ are all distinct, as claimed. 
But then applying \Cref{claim:forest of stars characterization:weight of ISF is bigger than sum of gains} on $F \coloneqq G[\{v, w, x, t\}]$, we get
\[4 < 6f_A(3) + 2f_A(2) - \gamma(y') - \gamma(s') = 3 + \frac 65 - 2 \cdot \frac 1{10} = 4,\]
which is a contradiction. 

{\bf Case 2: $st \notin E(G)$.}
Let $x'$ be the neighbor of $x$ distinct from $v$ and $y$.
Let $G' \coloneqq G \setminus \{v, w, s, x, y, x'\}$ and let $(A', B')$ be the restriction of $(A, B)$ to $V(G')$ where $t$ is moved to $B'$.
By minimality of our counterexample, we know that there exists an induced forest of stars $F'$ in $G'$ respecting $(A', B')$ and whose number of vertices is at least $f(G', A', B')$.
Define $F \coloneqq G[V(F') \cup \{v, x, w\}]$.
Observe that $F$ is an induced forest of stars in $G$ respecting $(A, B)$ (thanks in part to the fact that $t \in B'$). 
Let $Z \coloneqq N(\{v, w, s, x, y, x'\}) \setminus \{t\}$.
Note that $\abs {V(F)} < f(G, A, B)$ since $G$ and $(A, B)$ form a counterexample.
Thus
\begin{align*}
f(G, A, B) &> \abs {V(F)} = \abs {V(F')}+3 \\
&\geq f(G', A', B') + 3 \\
&= f(G, A, B) + 3 - 2 \cdot \frac 35 - 5 \cdot \frac 12 + f(t; G') + \sum_{z \in Z}\left(f(z; G') - f(z; G)\right)
\end{align*}
and hence
\begin{align}\label{eq:claim:dist(v,w) >= 4:case 2}
\frac 7{10} > f(t; G') + \sum_{z \in Z}\left(f(z; G') - f(z; G)\right) \geq f(t; G') + \abs Z\frac 1{10}.
\end{align}
Observe that $d_{G'}(t) \leq d_G(t)-1=2$ since $tw \in E(G)$.

Let $Z_1 \coloneqq Z \cap N(x')$ and let $Z_2 \coloneqq Z \cap N(s)$.
We claim that $\abs {Z_2} = 2$.
To show this observe that:
\begin{itemize}
    \item $sw \notin E(G)$, since $s \neq t$;
    \item $sx \notin E(G)$ by \Cref{claim:forest of stars:no A2 in a triangle};
    \item $sy \notin E(G)$ by \Cref{claim:forest of stars:4 distinct vertices at distance 2 from an A2} on $v$;
    \item $sx' \notin E(G)$ by \Cref{claim:forest of stars:4 distinct vertices at distance 2 from an A2} on $v$ as well;
    \item $st \notin E(G)$ by the hypothesis of Case 2.
\end{itemize}
Altogether, this implies that $\abs {Z_2} = 2$.
Now observe that by \eqref{eq:claim:dist(v,w) >= 4:case 2}:
\[f(t; G') < \frac 1{10}\left(7 - \abs Z\right) \leq \frac 1{10}\left(7 - \abs {Z_2}\right) = \frac 12.\]
Thus, we deduce that $f(t; G') = \frac 13$ and $d_{G'}(t) = 2$. 
In particular, $N(t) \cap \{v, s, x, y, x'\} = \varnothing$.

Now we show that $\abs {Z_1} = 2$.
Observe that:
\begin{itemize}
    \item $x'v \notin E(G)$ by \Cref{claim:forest of stars:no A2 in a triangle};
    \item $x'w \notin E(G)$, since otherwise $x'=t$, which is not possible since $xt \notin E(G)$; 
    \item $x's \notin E(G)$ by \Cref{claim:forest of stars:4 distinct vertices at distance 2 from an A2} on $v$;
    \item $x'y \notin E(G)$ by \Cref{claim:forest of stars:no A2 in a triangle};      
    \item $x't \notin E(G)$ by the remark hereabove.
\end{itemize}
This implies that $\abs {Z_1} = 2$.

Furthermore, by \eqref{eq:claim:dist(v,w) >= 4:case 2}, we know that
\[\frac 1{10}\abs Z \leq \sum_{z \in Z}\gamma(z) \leq \sum_{z \in Z}\left(f(z; G') - f(z; G)\right) < \frac 7{10} - f(t; G') = \frac 7{10} - \frac 13 = \frac {11}{30}.\]
In particular, $\abs Z < \frac {110}{30}$, and thus $\abs Z \leq 3$. 
Since $Z_1$ and $Z_2$ are $2$-element subsets of $Z$, we deduce that there exists $z' \in Z_1 \cap Z_2$. 
Thus, $z'\in Z$ and $x', s \in N(z')$, and hence $f(z'; G')-f(z'; G) \geq \frac 13$. 
Using the fact that $\abs Z \geq 2$, it follows that 
\[
\sum_{z \in Z}\left(f(z; G')-f(z; G)\right) \geq \frac 13 + \frac 1{10} > \frac {11}{30},
\]
contradicting the inequality above. This concludes the proof of the claim.
\end{claimproof}

We are now ready for finishing the proof of the lemma. 
Clearly, $A_3 \neq \varnothing$ (possibly $A_2 = \varnothing$).  
For every vertex $v\in A_2$, let $P_v$ be the path $uvw$ where $u, w$ are the two neighbors of $v$ in $G$. 
Observe that $u, w\in A_3$, and $uw \notin E(G)$ by \Cref{claim:forest of stars:no A2 in a triangle}. 
Furthermore, $P_v$ and $P_{v'}$ are vertex disjoint for every two distinct vertices $v, v' \in A_2$ by \Cref{claim:forest of stars:distance between A2's is >= 4}. 
Let $G'$ be obtained from $G - A_2$ by adding an edge between the two endpoints of $P_v$ for every $v\in A_2$. 
(Equivalently, $G'$ can be seen as being obtained from $G$ by contracting one edge of each path $P_v$.) 
Observe that $G'$ is a cubic graph. 

Applying \Cref{lemma:alpha_1 geq |G|/2 if G is cubic} on $G'$ gives  an induced forest $F'$ in $G'$ with at least
$\frac 12\abs {V(G')} = \frac 12\abs {A_3}$ vertices, and whose components are all isomorphic to either $K_1$ or $K_2$.
Let $F \coloneqq G[V(F') \cup A_2]$. 
We claim that $F$ is an induced forest of stars in $G$. 
Clearly, $F$ is a forest, as follows from the definition of $G'$. 
Moreover, no connected component of $F$ can be incident to more than one vertex from $A_2$ (otherwise, two vertices of $A_2$ would be at distance less than $4$ in $G$, contradicting \Cref{claim:forest of stars:distance between A2's is >= 4}). 
By the definition of $G'$, it follows then that every connected component of $F$ containing a vertex $v$ of $A_2$ is a path on at most $3$ vertices (note that this could possibly be the path $P_v$ in case $F'$ contains the edge linking both endpoints of $P_v$ in $G'$).  
Furthermore, since $V(G) = A_2 \cup A_3$, $F$ trivially respects $(A, B)$.
Finally, observe that
\[\abs {V(F)} = \abs {V(F')}+\abs {A_2} \geq \frac 12\abs {A_3} + \abs {A_2} \geq \frac 12\abs {A_3} + \frac 35\abs {A_2} = f(G, A, B),\]
which contradicts the fact that $G$ and $(A, B)$ form a counterexample. 
This concludes the proof of the lemma. 
\end{proof}